\documentclass[cmp,draft,numbook]{svjour}
\usepackage{amsmath,amssymb,verbatim}

\def\mscname{{\bfseries Subject Classifications\runinend}}
\def\msc#1{\par\addvspace\baselineskip\noindent\mscname\enspace\ignorespaces#1}

\counterwithout{equation}{section}
\counterwithout{proposition}{section}
\counterwithout{theorem}{section}
\counterwithout{corollary}{section}
\counterwithout{lemma}{section}

\def\a{\mathfrak{a}}
\def\al{\alpha}

\def\be{\beta}
\def\bh{{\bar h\,}}

\def\bi{{\bar\imath\,}}
\def\bj{{\bar\jmath\,}}
\def\bk{{\bar k\,}}
\def\bl{{\bar l\,}}

\def\CC{{\mathbb{C}}}
\def\CN{\CC^{\ts N|N}}

\def\de{\delta}
\def\De{\Delta}

\def\End{\operatorname{End}\hskip1pt}
\def\ep{\varepsilon}

\def\g{\mathfrak{g}}
\def\ge{\geqslant}
\def\glMN{\mathfrak{gl}_{\ts M|N}}
\def\glN{\mathfrak{gl}_{\ts N|N}}
\def\GP{G^{\,\prime}}
\def\gr{\operatorname{gr}}
\def\grp{\operatorname{gr}^{\,\prime}}

\def\h{\mathfrak{h}}

\def\IC{\mathcal{I}}
\def\id{{\mathrm{id}}}
\def\im{{\mathrm{im}}}

\def\JC{\mathcal{J}}

\def\lc{{\ts,\hskip1pt\ldots\ts,\ts}}
\def\le{\leqslant}

\def\ns{{\hskip-.5pt}}

\def\om{\omega}

\def\ot{\otimes}

\def\pN{\mathfrak{p}_{N}}

\def\S{\operatorname{S}}
\def\SC{\mathcal{S}}
\def\Sg{\mathfrak{S}}
\def\si{\sigma}

\def\TA{T^{\,\ast\ns}}
\def\TC{\mathcal{T}}
\def\TP{T^{\,\prime}}
\def\ts{{\hskip.5pt}}

\def\Ug{\operatorname{U}(\ts\g\ns)}
\def\Uh{\operatorname{U}(\ts\h\ns)}
\def\UpN{\operatorname{U}(\ts\pN\ns)}

\def\XP{X^{\,\prime}}
\def\XpN{\operatorname{X}(\ts\pN\ns)}

\def\YN{\operatorname{Y}(\ts\mathfrak{gl}_{\ts N}\ns)}
\def\YpN{\operatorname{Y}(\ts\pN\ns)}

\def\ZpN{\operatorname{Z}(\ts\pN\ns)}
\def\ZZ{\mathbb{Z}}

\begin{document}

\title{Yangian of the periplectic Lie superalgebra}
\titlerunning{Yangian of the periplectic Lie superalgebra}

\author{Maxim Nazarov}
\authorrunning{Maxim Nazarov}

\institute{Department of Mathematics, University of York, 
York YO10 5DD, United Kingdom}

\date{}

\maketitle

%------------------------------------------------------------------------------

\thispagestyle{empty} % NO FIRST PAGE NUMBER

\begin{abstract}
We study in detail the Yangian of the periplectic Lie superalgebra.
For this Yangian  we verify an analogue 
of the Poincar\'e\ts-Birkhoff\ts-\ns Witt Theorem. 
Moreover
we introduce a family of free generators of the centre of this Yangian.

\end{abstract}

\keywords{Hopf algebra, Lie superalgebra, Yangian}

\msc{16T20, 17A70, 17B37}

\newpage%%%%%%%%%%%%%%%%%%%%%%%%%%%%%%%%%%%%%%%%%%%%%%%%%%%%%%%%%%%%%%%%%%%%%%%

%==============================================================================

%\vspace{-\baselineskip}
\refstepcounter{section}
\subsection*{\bf\thesection}
\label{1}

Let $\glMN$ be the general linear Lie superalgebra over the complex field
$\CC\,$. Here both $M$ and $N$ are positive integers.
The Yangian of $\glMN$ was introduced in \cite{N1}
by extending the definition of the Yangian $\YN$
of the general linear 
Lie algebra $\mathfrak{gl}_{\ts N}\,$, 
see for instance \cite{N4}. 
The Yangian of $\glMN$ is a deformation of
the universal enveloping algebra 
of the polynomial current 
Lie superalgebra $\glMN[u]$ in the class of Hopf algebras,
see \cite{N5} for further details.

\vspace{1pt}

Let $\pN$ be the periplectic Lie superalgebra over %the field 
$\CC\,$.
This is a subalgebra of %the Lie superalgebra 
$\glN$ preserving
a non-degenerate odd symmetric bilinear form on the
$\ZZ_2$-graded vector space
$\CN$. The Lie superalgebra $\glN$ acts on $\CN$ in the natural~way.
Equivalently, the Lie superalgebra $\pN$ can be defined
as a fixed point subalgebra of $\glN$ with respect to a certain involutive
automorphism. This automorphism is denoted by $\om\,$,
see Section \ref{3} for an explicit definition of $\pN$ in these terms. 

The corresponding Yangian $\YpN$ was briefly introduced in \cite{N2}. 
Note that the construction of $\YpN$ therein involved 
the {\it periplectic Brauer algebra\/} which was later studied in by other authors, see for instance \cite{CP,C2,KT,M}.

According to the general scheme of \cite{D},
the Yangian of the Lie superalgebra $\pN$ cannot be defined
as a deformation of the universal enveloping algebra 
of the polynomial current 
Lie superalgebra $\pN[u]$ in the class of Hopf algebras.
This is because by \cite{S} the only 
supersymmetric $\pN$-invariant element of $\pN^{\ts\ot2}$ is zero. 
Consequently, there is no natural Lie co-superalgebra structure on 
$\pN[u]\,$.

Remarkably, there is
a natural Lie co-superalgebra structure on 
the {\it twisted\/} polynomial current Lie superalgebra
$
\g=\{\,g\ts(u)\in\glN[u]:\om\ts(\ts g\ts(u))=g\ts(-\ts u)\ts\}\,.
$
The Yangian $\YpN$ is a deformation of the universal enveloping algebra
$\Ug$ in the class of Hopf algebras.
See \cite{N2} for further discussion of this fact.

The definition of the Yangian $\YpN$ is given in our Section \ref{10}.
It is based on a new solution of the quantum Yang-Baxter equation
constructed in \cite{N2}. This solution is a rational function of two
variables $u\,,\ns v$ with values in the periplectic Brauer algebra. 
Unlike other rational solutions discovered before \cite{N2}
this is {\it not\/} a function of only the difference $u-v$
of the variables, see our Section~\ref{5}.

Our Theorem \!\ref{T1} is stated in \cite{N2} without proof. 
We prove it here. We also prove
an analogue for $\YpN$ of the 
Poincar\'e\ts-Birkhoff\ts-\ns Witt 
Theorem for~$\Ug\,$,~see \cite{MM}.
This analogue is our Theorem \ref{T2} which was also stated
in \cite{N2} without proof. We sketch its proof here.
Its Corollary \ref{C1}
describes $\YpN$ as a vector space explicitly.

Unlike its analogues \cite{AMR} 
for the Yangians of the symplectic and orthogonal Lie algebras,
Theorem 2 cannot be proved by using the method of \cite{N3}.
The reason for it is explained in %at the end of 
our Section \ref{11}. We use the Diamond Lemma of~\cite{B}~instead. 

The Lie superalgebra $\pN$ is a subalgebra of $\g\,$.
It consists of all those elements of $\g$ which do not depend on 
the variable $u\,$. So the universal enveloping algebra $\UpN$
is a Hopf subalgebra of $\Ug\,$.
We also have an embedding of $\operatorname{U}(\ts\pN\ts)$ 
to~the Yangian $\YpN$ as a Hopf subalgebra, see Corollary \ref{C2}
to our Theorem \ref{T2}.

There also is a subalgebra of $\g$ isomorphic to 
$\mathfrak{gl}_{\ts N}[u]\,$.
It consists of all 
$g(u)\in\g$ taking values only in the even part of $\glN\,$.
Hence there is a Hopf subalgebra of $\operatorname{U}(\ts\g\ts)$
isomorphic to %the universal enveloping algebra 
$\operatorname{U}(\ts\mathfrak{gl}_{\ts N}[u]\ts)\,$.
We have an embedding of $\YN$ to $\YpN$ 
as an associative subalgebra, see Corollary~\ref{C3}.
But it is not a Hopf algebra embedding.

Our Theorem \ref{T3} yields an explicit 
description of the centre of $\YpN\,$. We use an analogue for
$\YpN$ of the {\it quantum Berezinian\/} \cite{N1}
for the Yangian of $\glMN\,$. To define it we
construct a family of one-dimensional subquotient
representations of the algebra $\YpN$ depending on a complex parameter,
see Sections \ref{15} and~\ref{25}.

\enlargethispage{36pt}%%%%%%%%%%%%%%%%%%%%%%%%%%%%%%%%%%%%%%%%%%%%%%%%%%%%%%%%%

Irreducible finite-dimensional representations of the Lie superalgebra 
$\pN$ were recently studied by several authors, 
see for instance \cite{AGG1,AGG2,BB,C1,DLZ,ES,IRS}. 
It would be %very 
interesting to study irreducible finite-dimensional
representations of %the Yangian 
$\YpN\ts$. For the Yangian of %the general linear Lie superalgebra 
$\glMN$ this was done in \cite{Z} and more recently in \cite{L,LM,T}.

%==============================================================================

%\vspace{-\baselineskip}
\refstepcounter{section}
\subsection*{\bf\thesection}
\label{2}

We shall use the following general conventions.
Let $\mathrm{A}$~and~$\mathrm{B}$ be %any two 
associative
$\ZZ_{\ts2}$-graded algebras. Their tensor product
$\mathrm{A}\ot\mathrm{B}$ is also
an associative $\ZZ_2$-graded algebra such that for any homogeneous
elements
$X,X^\prime\in\mathrm{A}$ and $Y,Y^\prime\in\mathrm{B}$
\begin{align*}
%\label{conv1}
(X\ot Y)\ts (X^\prime\ot Y^\prime)&=X\ts X^\prime\ot Y\ts Y^\prime
\,(-1)^{\ts\deg X^\prime\deg Y},
\\
%\label{conv2}
\deg\ts(X\ot Y)&=\deg X+\deg Y\ts.
\end{align*}
Furthermore, for any two $\ZZ_2$-graded modules
$U$ and $V$ over $\mathrm{A}$ and $\mathrm{B}$ respectively,
the vector space $U\ot V$ is a $\ZZ_2$-graded module over 
$\mathrm{A}\ot\mathrm{B}$ such that for any homogeneous elements
$x\in U$ and $y\in V$
\begin{align}
\label{XY}
(X\ot Y)\ts (\ts x\ot y\ts )&=X\ts x\ot Y\ts y
\,(-1)^{\ts\deg x\,\deg Y},
\\
\label{xy}
\deg\hspace{0.5pt}(\ts x\ot y\ts)&=\deg x+\deg y\,.
\end{align}
A homomorphism $\al:\mathrm{A}\to\mathrm{B}$ is a linear map
such that $\al\ts(X\,X^\prime)=\al\ts(X)\,\al\ts(X^\prime)$ for all 
$X,X^\prime\in\mathrm{A}\,$. But an antihomomorphism
$\be:\mathrm{A}\to\mathrm{B}$ is a linear map
such that for all homogeneous $X,X^\prime\in\mathrm{A}$ 
\begin{equation}
\label{bb}
\be\ts(X\ts X^\prime\ts)=
\be\ts(X^\prime)\,\be\ts(X)\,(-1)^{\ts\deg X\deg X^\prime}.
\end{equation}

Let $n$ be any positive integer.
If the algebra $\mathrm{A}$ is unital, let $\iota_{\ts p}$ be 
its embedding into the tensor product~$\mathrm{A}^{\ns\ot\ts n}$
as the $p\ts$-th tensor factor\ts:
\begin{equation*}
\iota_{\ts p}\ts(X)=1^{\ot\ts (p-1)}\ot X\ot1^{\ot\ts(n-p)}
\quad\text{for}\quad
p=1\lc n\,.
\end{equation*}
We will also use various embeddings of %the algebra
$\mathrm{A}^{\ns\ot\ts m}$ into $\mathrm{A}^{\ns\ot\ts n}$
for $m=1\lc n$.
For any choice of $m$ pairwise distinct indices 
$p_1\lc p_m\in\{\ts1\lc n\ts\}$
and of an element $X\in\mathrm{A}^{\ns\ot\ts m}$ of the form
$X=X^{(1)}\ot\ldots\ot X^{(m)}$ we will denote
\begin{equation*}
X_{\ts p_1\ldots\ts p_m}=\ts
\iota_{\ts p_1}(X^{(1)})\ts\ldots\,\ts\iota_{\ts p_m}(X^{(m)})
\in\mathrm{A}^{\ns\ot\ts n}.
\end{equation*}
We will then extend the notation $X_{\ts p_1\ldots\ts p_m}$ to all elements
$X\in\mathrm{A}^{\ns\ot\ts m}$ by linearity.

%------------------------------------------------------------------------------

\vspace{-\baselineskip}
\refstepcounter{section}
\subsection*{\bf\thesection}
\label{3}

Let the indices $i\,,j\,,k\,,l$ run through $\ts\pm1\lc\pm N\,$.
Put $\bi=0$ if $i>0$ and $\bi=1$ if $i<0\,$. 
Now consider the
$\ZZ_{\ts2}$-graded vector space $\CN\ts$.
Let $e_i\in\CN$ be an element of the standard basis.
The $\ZZ_{\ts2}$-grading on $\CN$ is defined by $\deg e_i=\bi\,$.
We will be using the bilinear form $\langle\ ,\,\rangle$ on %the vector space
$\CN$ defined by setting $\langle\,e_i,e_j\,\rangle=\de_{\,i\ts,\ts-j}$
for any indices $i$ and $j\,$. This form is clearly symmetric.

Let $E_{\ts ij}\in\End\CN$ be the standard matrix unit, defined by
$E_{\ts ij}\,e_{\ts k}=\de_{\ts jk}\,e_{\ts i}\,$. The 
associative algebra $\End\CN$
is $\ZZ_{\ts2}$-graded so that $\deg E_{\ts ij}=\bi+\ts\bj\,$.
Hence $\CN$ is a $\ZZ_2$-graded module over $\End\CN\,$. 
For any $n$ we can also identify
the tensor product $(\End\CN)^{\ot\ts n}$ with the algebra
$\End((\CN)^{\ot\ts n})$ acting on the vector space $(\CN)^{\ot\ts n}$ 
by repeatedly using the conventions \eqref{XY} and \eqref{xy}.

We can define an antiautomorphism $\tau$ of $\End\CN$ by mapping
\begin{equation*}
\tau:\,E_{\ts ij}\mapsto E_{\ts ji}\,{(-1)}^{\,\bi\ts\bj\ts+\,\bi}\,.
\end{equation*}
Note that $\tau$ is not involutive, while
$\tau^{\ts2}$ is the {\it parity automorphism\/} of $\End\CN$
\begin{equation}
\label{parity}
E_{\ts ij}\mapsto E_{\ts ij}\,{(-1)}^{\,\bi\ts+\ts\bj}\ts.
\end{equation}

We can also define an involutive automorphism $\pi$ of $\End\CN$ by mapping
\begin{equation*}
\pi:\,E_{\ts ij}\mapsto E_{\,-i\ts,\ts-j}\,.
\end{equation*}
The compositions $\tau\,\pi$ and $\pi\,\tau$ again differ by 
the automorphism \eqref{parity}. Hence both $\tau\,\pi$ and $\pi\,\tau$ are involutive 
antiautomorphisms of $\End\CN\,$. Explicitly,
$$
\tau\,\pi:\,E_{\ts ij}\mapsto E_{\,-j\ts,\ts-i}\,
{(-1)}^{\,\bi\ts\bj\ts+\,\bj}\ts.
$$

Consider the Lie superalgebra $\glN\,$.
To avoid confusion, denote by $e_{\ts ij}$ 
the element of $\glN$ corresponding to $E_{\ts ij}\in\End\CN\,$. 
Then  $\deg e_{\ts ij}=\bi+\ts\bj$ and %the Lie bracket
\begin{equation}
\label{glN}
[\,e_{\ts ij}\,,e_{\ts kl}\,]=
\de_{jk}\,e_{\ts il}-
\de_{\ts li}\,e_{\ts kj}\,
{(-1)}^{\ts(\,\bi+\,\bj\,)(\ts\bk+\,\bl\,)}\,.
\end{equation}
It follows that both $\pi$ and $-\,\tau$ are automorphisms of the 
Lie superalgebra $\glN\,$. Let $\om=-\,\tau\,\pi$ be their composition.
The automorphism $\om$ %of $\glN$ 
is involutive. Further, 
$$
\langle\,Z\,x\ts,y\,\rangle+\langle\,x\ts,\om\ts(Z)\,y\,\rangle
(-1)^{\ts\deg x\deg Z}=0
$$
for any homogeneous $x\ts,y\in\CN$ and $Z\in\glN\,$.
The Lie superalgebra $\glN$ acts on the vector space $\CN$
via the above identification with $\End\CN\,$.

Now the {\it periplectic Lie superalgebra} $\pN$ is
the fixed point subalgebra of $\glN$ with respect to
the automorphism $\om\,$. It can also be defined as 
the subalgebra of $\glN$ preserving the form $\langle\ ,\,\rangle$ 
on $\CN\,$. This subalgebra is spanned by the elements
\begin{equation}
\label{fij}
f_{\ts ij}=e_{\ts ij}+\om\ts(e_{\ts ij})=
e_{\ts ij}-e_{\ts-j\ts,\ts-i}\,
{(-1)}^{\,\bi\ts\bj\ts+\,\bj}\ts.
\end{equation}

%------------------------------------------------------------------------------

\vspace{-\baselineskip}
\refstepcounter{section}
\subsection*{\bf\thesection}
\label{4}

Take the element of the algebra $(\End\CN)^{\ot\ts2}$
\begin{equation}
\label{p}
P=\sum_{i,j}\,E_{\ts ij}\ot E_{\ts ji}\,(-1)^{\,\bj}\ts.
\end{equation}
It acts on the vector space $(\CN)^{\ts\ot\ts2}$ so that
$
e_i\ot e_{j}\mapsto e_j\ot e_i\,{(-1)}^{\,\bi\ts\bj}\ts.
$
Here we identify the algebra $(\End\CN)^{\ot\ts2}$ with the algebra 
$\End((\CN)^{\ts\ot\ts2})$ using~\eqref{XY}. Note that $P^{\ts2}=1\,$.
Also note that
$
(\ts\tau\,\pi\ot\tau\,\pi\ts)(P)=-\,P\,.
$
Now let 
\begin{equation}
\label{pt}
Q=(-\,\tau\,\pi\ot\id\,)(P)=(\,\id\ot\tau\,\pi\ts)(P)\,.
\end{equation}
Explicitly,
\begin{equation}
\label{q}
Q=\sum_{i,j}\,E_{\ts ij}\ot E_{\,-i\ts,\ts-j}\,
{(-1)}^{\,\bi\ts\bj\ts+\ts\bi\ts+\ts\bj}\ts.
\end{equation}
The image of the action of $Q$ on %the vector space
$(\CN)^{\ts\ot\ts2}$ is one-dimensional and is spanned~by %the vector
\begin{equation}
\label{qimage}
\sum_i\,e_{\ts i}\ot e_{\,-\ts i}\,{(-1)}^{\,\bi}\,.
%\vspace{-2pt}
\end{equation}
Here we regard $Q$ as an element of %the algebra 
$\End((\CN)^{\ts\ot\ts2})$
by once again identifying~the latter algebra 
with $(\End\CN)^{\ts\ot\ts2}$ using \eqref{XY}.
Also note that 
\begin{equation}
\label{pq}
P\,Q=-\,Q\,,
\quad
Q\,P=Q
\quad\text{and}\quad
Q^{\ts2}=0\,.
\end{equation}
By \cite[Theorem 4.5]{M}
the supercommutant of the image of $\pN$ in %the $\ZZ_{\ts2}$-graded algebra
$(\End\CN)^{\ot\ts n}$ for any $n$
is generated by all elements $P_{\,pq}$ and $Q_{\,pq}$ where $1\le p<q\le n\,$.
Here we use the standard comultplication
on the universal enveloping algebra of $\pN\,$.

%------------------------------------------------------------------------------

%\vspace{-\baselineskip}
\refstepcounter{section}
\subsection*{\bf\thesection}
\label{5}

Consider a function of complex variables $u\,,\ns v$
with values in %the algebra
$(\End\CN)^{\ot\ts2}$ 
\begin{equation*}
%\label{R}
R\ts(u\,,\ns v)=1-\frac{P}{u-v}+\frac{Q}{u+v}\,.
\end{equation*}
By \eqref{pt} we have
\vspace{-6pt}
\begin{align}
(\tau\,\pi\ot1)(R(u\,,\ns v))&=R(u\,,-\ts v)\,,
\label{t1}
\\[4pt]
(1\ot\tau\,\pi)(R(u\,,\ns v))&=R(-\ts u\,,\ns v)\,.
\label{t2}
\end{align}
We also have
\begin{equation}
\label{rr}
R\ts(u\ts,\ns v)\,R\ts(-\ts u\,,-\ts v)=1-\frac1{(u-v)^{\ts2}}\,.
\end{equation}
Indeed, due to the relation $P^{\ts2}=1$ and to \eqref{pq}
\begin{gather*}
\left(1-\frac{P}{u-v}+\frac{Q}{u+v}\right)
\left(1+\frac{P}{u-v}-\frac{Q}{u+v}\right)=
\\[4pt]
1-\frac{P^{\ts2}}{(u-v)^{\ts2}}+\frac{P\,Q+Q\,P}{(u-v)(u+v)}
-\frac{Q^{\ts2}}{(u+v)^{\ts2}}=1-\frac1{(u-v)^{\ts2}}\,.
\end{gather*}
%The main property of the function $R\ts(u\,,\ns v)$ is stated as the
%following theorem.
\vspace{-4pt}

\begin{theorem}
\label{T1}
The rational function $R\ts(u\,,\ns v)$ satisfies the {Yang-Baxter equation}
in the algebra $(\End\CN)^{\ot\ts 3}(\ts u,v,w\ts)$
\begin{equation}
\label{rrr}
R_{\ts12}(u\,,\ns v)\ts\ts R_{\ts13}(u\,,\ns w)\ts\ts R_{\ts23}(v\,,\ns w)=
R_{\ts23}(v\,,\ns w)\ts\ts R_{\ts13}(u\,,\ns w)\ts\ts R_{\ts12}(u\,,\ns v)\,.
\end{equation}
\end{theorem}

\begin{proof}
Using the definition of $R\ts(u\,,\ns v)$
the equality in \eqref{rrr} will follow from the %fifteen groups of
relations in the algebra $(\End\CN)^{\ot\ts 3}$ displayed below\ts:
\begin{gather}
\label{R1}
P_{\ts12}\,P_{\ts13}=P_{\ts23}\,P_{\ts12}=P_{\ts13}\,P_{\ts23}\,,
\\[4pt]
\label{R2}
P_{\ts13}\,P_{\ts12}=P_{\ts12}\,P_{\ts23}=P_{\ts23}\,P_{\ts13}\,,
\\[4pt]
\label{R3}
Q_{\ts13}\,Q_{\ts12}=-\,P_{\ts23}\,Q_{\ts12}=-\,Q_{\ts13}\,P_{\ts23}\,,
\\[4pt]
\label{R4}
Q_{\ts12}\,Q_{\ts13}=-\,Q_{\ts12}\,P_{\ts23}=-\,P_{\ts23}\,Q_{\ts13}\,,
\\[4pt]
\label{R5}
Q_{\ts12}\,P_{\ts13}=-\,Q_{\ts12}\,Q_{\ts23}=-\,P_{\ts13}\,Q_{\ts23}\,,
\\[4pt]
\label{R6}
P_{\ts13}\,Q_{\ts12}=-\,Q_{\ts23}\,Q_{\ts12}=-\,Q_{\ts23}\,P_{\ts13}\,,
\\[4pt]
\label{R7}
P_{\ts12}\,Q_{\ts13}=Q_{\ts23}\,P_{\ts12}=Q_{\ts23}\,Q_{\ts13}\,,
\\[4pt]
\label{R8}
Q_{\ts13}\,P_{\ts12}=P_{\ts12}\,Q_{\ts23}=Q_{\ts13}\,Q_{\ts23}\,,
\\[4pt]
\label{S1}
P_{\ts12}\,P_{\ts13}\,P_{\ts23}=P_{\ts23}\,P_{\ts13}\,P_{\ts12}\,,
\\[4pt]
\label{S6}
P_{\ts12}\,Q_{\ts13}\,Q_{\ts23}=Q_{\ts23}\,Q_{\ts13}\,P_{\ts12}\,,
\\[4pt]
\label{S7}
Q_{\ts12}\,Q_{\ts13}\,P_{\ts23}=P_{\ts23}\,Q_{\ts13}\,Q_{\ts12}\,,
\\[4pt]
\label{S2}
Q_{\ts12}\,P_{\ts13}\,Q_{\ts23}=0\,,
\\[4pt]
\label{S3}
Q_{\ts23}\,P_{\ts13}\,Q_{\ts12}=0\,,
\\[4pt]
\label{S4}
P_{\ts12}\,P_{\ts13}\,Q_{\ts23}=
P_{\ts23}\,Q_{\ts13}\,P_{\ts12}=
Q_{\ts12}\,P_{\ts13}\,P_{\ts23}=
-\,Q_{\ts12}\,Q_{\ts13}\,Q_{\ts23}\,,
\\[4pt]
\label{S5}
Q_{\ts23}\,P_{\ts13}\,P_{\ts12}=
P_{\ts12}\,Q_{\ts13}\,P_{\ts23}=
P_{\ts23}\,P_{\ts13}\,Q_{\ts12}=
-\,Q_{\ts23}\,Q_{\ts13}\,Q_{\ts12}\,.
\end{gather}

To prove \eqref{rrr}, the relations \eqref{R1} and \eqref{R2}
are used along with the identity 
\begin{equation}
\label{R12}
\frac1{u-v}\,\frac1{u-w}-
\frac1{u-v}\,\frac1{v-w}+
\frac1{u-w}\,\frac1{v-w}=0
\end{equation}
which is easy to verify. 
The relations \eqref{R3} and \eqref{R4} are used with the identity
\begin{equation*}
\frac1{u+v}\,\frac1{u+w}+
\frac1{u+v}\,\frac1{v-w}-
\frac1{u+w}\,\frac1{v-w}=0
\end{equation*}
which is obtained from \eqref{R12} by changing the sign of $u\,$.
The relations \eqref{R5}~and \eqref{R6} are used with the identity
\begin{equation*}
\frac1{u+v}\,\frac1{u-w}+\frac1{u+v}\,\frac1{v+w}-
\frac1{u-w}\,\frac1{v+w}=0
\end{equation*}
which is obtained from \eqref{R12} by changing the sign of $v\,$.
The relations \eqref{R7}~and \eqref{R8} are used with the identity
\begin{equation*}
\frac1{u-v}\,\frac1{u+w}-\frac1{u-v}\,\frac1{v+w}+
\frac1{u+w}\,\frac1{v+w}=0
\end{equation*}
which is obtained from \eqref{R12} by changing the sign of $w\,$.
Finally, the relations \eqref{S4} and \eqref{S5} are used 
along with another identity which is easy to verify\ts: 
\begin{align*}
\frac1{u-v}\,\frac1{u-w}\,\frac1{v+w}&-
\frac1{u-v}\,\frac1{u+w}\,\frac1{v-w}\,\ts+
\\[4pt]
\frac1{u+v}\,\frac1{u-w}\,\frac1{v-w}&-
\frac1{u+v}\,\frac1{u+w}\,\frac1{v+w}=0\,.
\end{align*}

Let us verify the relations \eqref{R1} to \eqref{R8}. 
The relations \eqref{R1} and \eqref{R2} follow from
the description of the action of $P$ on 
the vector space $(\CN)^{\ts\ot\ts2}$ as provided in the beginning of
Section~\ref{4}.
Due to \eqref{pt}, the relations \eqref{R3} and \eqref{R4}
can be obtained by applying
the antiautomorphism $\tau\,\pi$ of $\End\CN$ to the relations
\eqref{R1} and \eqref{R2} relative to the first tensor factor of
$(\End\CN)^{\ot\ts 3}$.
Similarly, the relations \eqref{R5} and \eqref{R6}
can be obtained by applying $\tau\,\pi$
to the relations
\eqref{R1} and \eqref{R2} relative to the second tensor factor of
$(\End\CN)^{\ot\ts 3}$. 
The relations \eqref{R7} and \eqref{R8}
can be obtained by applying $\tau\,\pi$ to the relations
\eqref{R1} and \eqref{R2} relative to the third tensor factor of
$(\End\CN)^{\ot\ts 3}$. 

Next let us verify the relations \eqref{S1} to \eqref{S5}. 
The above mentioned description of the action of $P$ on 
$(\CN)^{\ts\ot\ts2}$ shows
that either side of \eqref{S1} equals 
$P_{\ts13}\,$. This description also implies \eqref{S6} and \eqref{S7}.
It also implies that
$$
Q_{\ts12}\,P_{\ts13}\,Q_{\ts23}=
Q_{\ts12}\,Q_{\ts21}\,P_{\ts13}=
Q_{\ts12}\,P_{\ts12}\,Q_{\ts21}\,P_{\ts13}=
Q_{\ts12}\,Q_{\ts12}\,P_{\ts12}\,P_{\ts13}=0
$$
where we also used \eqref{pq}. So we get the relation~\eqref{S2}.
\text{Similarly, we get~\eqref{S3}\ts:} 
$$
Q_{\ts23}\,P_{\ts13}\,Q_{\ts12}=
P_{\ts13}\,Q_{\ts21}\,Q_{\ts12}=
-\,P_{\ts13}\,Q_{\ts21}\,P_{\ts12}\,Q_{\ts12}=
-\,P_{\ts13}\,P_{\ts12}\,Q_{\ts12}\,Q_{\ts12}=0\,.
$$
Also similarly, each of the first three products in \eqref{S4} is equal to
$-\,P_{\ts13}\,Q_{\ts23}\,$. By using the first equality in \eqref{R4}
and then the second equality in \eqref{R5},
$$
-\,Q_{\ts12}\,Q_{\ts13}\,Q_{\ts23}=
Q_{\ts12}\,P_{\ts23}\,Q_{\ts23}=
-\,Q_{\ts12}\,Q_{\ts23}=-\,P_{\ts13}\,Q_{\ts23}\,.
$$
Hence we get the last equality in \eqref{S4}.
The proofs of all three equalities~in~\eqref{S5} are very
similar to those in \eqref{S4} and are omitted here.
\qed 
\end{proof}

%------------------------------------------------------------------------------

%\vspace{-\baselineskip}
\refstepcounter{section}
\subsection*{\bf\thesection}
\label{6}

We will first define the {\it extended Yangian\/}
of the Lie superalgebra $\pN\,$. This~is a complex associative
unital %$\ZZ_{\ts2}$-graded 
algebra $\XpN$ with a countable set of generators 
\begin{equation}
\label{Tijr}
T^{\ts(r)}_{ij}
\quad\text{where}\quad
r=1,2,\,\ldots
\quad\text{and}\quad
i\,,j=\pm\,1\lc\!\pm N\,.
\end{equation}
The algebra $\XpN$ is $\ZZ_{\ts2}$-graded so that
$\deg T_{ij}^{\ts(r)}=\bi+\bj$ for all $r\,$.
To write~down defining relations for the
generators \eqref{Tijr} of $\XpN$ we will use the formal power
series in $u^{-1}$ with coefficients from $\XpN$
\begin{equation}
\label{3.0}
T_{ij}(u)=\de_{\ts ij}\cdot1+
T_{ij}^{\ts(1)}\ts u^{-1}+T_{ij}^{\ts(2)}\ts u^{-2}+\ldots\,.
\end{equation}

Now combine all the series \eqref{3.0} into a single element
\begin{equation}
\label{tu}
T(u)=\sum_{i,j}\,E_{\ts ij}\ot T_{ij}(u)
\in(\End\CN)\ot\XpN\ts[[u^{-1}]]\,.
\end{equation}
For any $n$ and any $p=1\lc n$ we will denote
\begin{equation}
T_p(u)=(\iota_p\ot\id)\ts(\ts T(u))
\in(\End\CN)^{\ot\ts n}\ot\XpN\,[[u^{-1}]]\,.
\label{3.22}
\end{equation}
By using this notation for $n=2$ the defining relations of $\XpN$ are
\begin{equation}
\label{3.3}
(R\ts(u\,,\ns v)\ot1)\,\ts T_1(u)\,T_2(v)=
T_2(v)\,T_1(u)\,(R\ts(u\,,\ns v)\ot1)\,.
\end{equation}
Namely, after multiplying each side of \eqref{3.3} by $u^{\ts2}-v^{\ts2}$ 
it becomes an equality~of formal Laurent series in $u^{-1},v^{-1}$
with coefficients in %the algebra
$(\End\CN)^{\ot\ts2}\ot\XpN\,$.
By expanding the relation \eqref{3.3} 
in the basis of $(\End\CN)^{\ot\ts2}$ consisting of
$$
E_{\ts ij}\ot E_{\ts kl}\,
{(-1)}^{\ts(\,\bi+\,\bj\,)(\ts\bk+\,\bl\,)}
\quad\text{where}\quad
i\,,j\,,k\,,l=\pm\,1\lc\!\pm N
$$
while using the definitions \eqref{p} and \eqref{q},
we get a collection of relations 
\begin{gather}
[\,T_{ij}(u)\ts,T_{kl}(v)\,]
=\frac{\,T_{kj}(u)\,T_{il}(v)-T_{kj}(v)\,T_{il}(u)}{u-v}\,
(-1)^{\ts\,\bi\ts\bk\ts+\,\bi\ts\bl\,+\,\bk\ts\bl}
\notag
\\[4pt]
-\,\,\de_{\ts -i,\ts k}\,\sum_h\,
\frac{\,T_{\ts hj}(u)\,T_{\ts -h,\ts l}(v)}{u+v}\,
(-1)^{\ts\,\bh\bl\ts+\,\bi\ts\bl\ts+\,\bi}
\notag
\\
+\,\,\de_{\ts j,-l}\,\,\sum_h\,
\frac{\,T_{k,-h}(v)\,T_{ih}(u)}{u+v}\,
(-1)^{\ts\,\bh\bk\ts+\,\bi\ts\bk\ts+\,\bi\ts\bl+\,\bh}
\label{rttttr}
\end{gather}
where $h=\pm1\lc\pm N$ and the square brackets stand for the supercommutator.

By comparing \eqref{rrr} with \eqref{3.3},
we obtain that for any $t\in\CC$ the assignment
\begin{equation*}
%\label{3.66}
(\End\CN)\ot\XpN\ts[[u^{-1}]]\to(\End\CN)^{\ot\ts2}\ts[[u^{-1}]]:\, 
T(u)\mapsto R(u\,,t\ts)
\end{equation*}
defines a representation $\XpN\to\End\CN$. We will denote
this representation by $\rho_{\,t}\,$.
By the definitions \eqref{p} and \eqref{q}, %under this representation 
for any $r\ge0$
\begin{equation*}
%\label{rhoz}
\rho_{\,t}:T_{ij}^{\ts(r+1)}\mapsto
(-\ts E_{\ts ji}+(-1)^{\ts r}\ts\tau\,\pi\ts(E_{\ts ji}))\,
t^{\,r}\ts{(-1)}^{\,\bj}\ts.
\end{equation*}

%------------------------------------------------------------------------------

%\vspace{-\baselineskip}
\refstepcounter{section}
\subsection*{\bf\thesection}
\label{7}

Due to \eqref{3.3} for any formal power series $c\ts(u)$ in $u^{-1}$
with coefficients from $\CC$ and leading term $1\ts$,
an automorphism of the algebra $\XpN$ can be defined by %mapping
\begin{equation}
\label{muf}
T_{ij}(u)\mapsto c\ts(u)\,T_{ij}(u)\,.
\end{equation}

Now for all $i$ and $j$ define a 
formal power series $\TP_{ij}(u)$ in $u^{-1}$
with coefficients in the algebra $\XpN$
by using the element inverse to \eqref{tu}\ts:
\begin{equation}
\label{Tui}
T(u)^{-1}=\sum_{i,j}\,E_{\ts ij}\ot\TP_{ij}(u)\,.
\end{equation}

\begin{proposition}
\label{P1}
These mappings define two commuting 
antiautomorphisms of the algebra\/ $\XpN\!:$
\vspace{-8pt}
\begin{align}
\label{M}
T_{ij}(u)&\mapsto T_{ij}(-\ts u)\,,
\\[2pt]
\label{S}
T_{ij}(u)&\mapsto\TP_{ij}(u)\,.
\end{align}
\end{proposition}

\begin{proof}
%Due to the convention \eqref{conv1} 
By using the notation \eqref{3.22} for $n=2$ we get the equalities
\begin{align}
\label{TT}
T_1(u)\,T_2(v)&=\sum_{i,j,k,l}\,
E_{\ts ij}\ot E_{\ts kl}\ot T_{ij}(u)\,T_{kl}(v)\,
{(-1)}^{\ts(\,\bi+\,\bj\,)(\ts\bk+\,\bl\,)}\,,
\\
\label{MM}
T_2(-v)\,T_1(-u)&=\sum_{i,j,k,l}\,
E_{\ts ij}\ot E_{\ts kl}\ot T_{kl}(-v)\,T_{ij}(-u)\,.
\end{align}
Due to \eqref{bb}
the antihomomorphism property of \eqref{M} follows from the relation
\begin{equation*}
(R\ts(u\,,\ns v)\ot1)\,T_2(-v)\,T_1(-u)=
T_1(-u)\,T_2(-v)\,(R\ts(u\,,\ns v)\ot1)
\end{equation*}
which %in turn 
is obtained from \eqref{3.3} by using \eqref{rr}.
The antihomomorphism~\eqref{M} is clearly involutive and therefore bijective.

Similarly to \eqref{MM}, by again using
the notation \eqref{3.22} for $n=2$ we get
\begin{equation*}
T_2(v)^{-1}\,T_1(u)^{-1}=\sum_{i,j,k,l}\,
E_{\ts ij}\ot E_{\ts kl}\ot\TP_{kl}(v)\,\TP_{ij}(u)\,.
\end{equation*}
By comparing this with \eqref{TT},
the antihomomorphism property of \eqref{S} follows from the relation
$$
(R\ts(u\,,\ns v)\ot1)\ts\,T_2(v)^{-1}\,T_1(u)^{-1}=
T_1(u)^{-1}\,T_2(v)^{-1}\ts(R\ts(u\,,\ns v)\ot1)
$$
which %in turn 
is obtained by multiplying both sides of the defining relation \eqref{3.3} 
on the left and right by $T_2(v)^{-1}$ and then by $T_1(u)^{-1}$.

The antihomomorphism \eqref{S} clearly commutes with \eqref{M}.
The bijectivity of the antihomomorphism \eqref{S}
will follow from Proposition \ref{P5} below. 
\qed
\end{proof}

Furthermore, let us introduce the element of 
$(\End\CN)\ot\XpN\ts[[u^{-1}]]$
\begin{gather*}
\TA(u)=(\ts\tau\,\pi\ot\id\,)(T(u))=
\\[6pt]
\sum_{i,j}\,E_{\,-j\ts,\ts-i}\ot T_{ij}(u)\,
{(-1)}^{\,\bi\ts\bj\ts+\,\bj}
=\sum_{i,j}\,E_{\ts ij}\ot T_{\,-j\ts,\ts-i}\ts(u)\,
{(-1)}^{\,\bi\ts\bj\ts+\,\bj}\ts.
\end{gather*}

\begin{proposition}
\label{P2}
An automorphism of\/ $\XpN$ commuting with \eqref{M} 
is defined by %mapping
\begin{equation}
\label{T}
T_{ij}(u)\mapsto 
T_{\,-j\ts,\ts-i}\ts(u)\,
{(-1)}^{\,\bi\ts\bj\ts+\,\bj}\ts.
\end{equation}
\end{proposition}

\newpage%%%%%%%%%%%%%%%%%%%%%%%%%%%%%%%%%%%%%%%%%%%%%%%%%%%%%%%%%%%%%%%%%%%%%%%

\begin{proof}
The composition $\tau\,\pi$ is an antiautomorphism of the $\ZZ_2$-graded 
associative algebra $\End\CN\,$. Therefore by applying the map 
$\tau\,\pi\ot\tau\,\pi\ot\id$ to both sides of \eqref{3.3} 
and then using \eqref{t1},\eqref{t2} we get the relation
\begin{equation}
\label{tata}
\TA_1(u)\,\TA_2(v)\,(R\ts(-\ts u\,,-\ts v)\ot1)=
(R\ts(-\ts u\,,-\ts v)\ot1)\,\TA_2(v)\,\TA_1(u)\,.
\end{equation}
Here for $p=1,2$ we denote
\begin{equation*}
\TA_p(u)=\iota_p\ot\id\,(\ts\TA(u))
\in(\End\CN)^{\ot\ts 2}\ot\XpN\,[[u^{-1}]]\,.
\end{equation*}
Multiplying \eqref{tata} on the left and right by
$R\ts(u\,,\ns v)\ot1$ and using \eqref{rr} yields
\begin{equation*}
(R\ts(u\,,\ns v)\ot1)\,\TA_1(u)\,\TA_2(v)=
\TA_2(v)\,\TA_1(u)\,(R\ts(u\,,\ns v)\ot1)
\end{equation*}
which yields the homomorphism property of \eqref{T}.
This homomorphism is clearly involutive, hence bijective.
It clearly commutes with \eqref{M} as well.
\qed
\end{proof}

%------------------------------------------------------------------------------

\vspace{-\baselineskip}
\refstepcounter{section}
\subsection*{\bf\thesection}
\label{8}
 
Let us multiply the %defining 
relation \eqref{3.3}
by $u+v$ and then set $v=-\ts u\,$. We get %the relation 
\begin{equation}
\label{QQ}
(Q\ot1)\,\ts T_1(u)\,T_2(-\ts u)=
T_2(-\ts u)\,T_1(u)\,(Q\ot1)\,.
\end{equation}
Since the image of the action of $Q$ on $(\CN)^{\ts\ot\ts2}$
is one-dimensional, either side of the last displayed relation
must be equal to $Q\ot Z(u)$ where $Z(u)$ is a power series 
in $u^{-1}$ with coefficients from $\XpN\,$.
It is immediate %from these equalities 
that every coefficient of the series $Z(u)$ has $\ZZ_{\ts2}$-degree $0$ and
that the leading term of $Z(u)$ is $1\,$. Hence
$$
Z(u)=1+Z^{\ts(1)}\ts u^{-1}+Z^{\ts(2)}\ts u^{-2}+\ldots\,.
$$

\begin{proposition}
\label{P3}
The elements $Z^{\ts(1)},Z^{\ts(2)}\ldots$ of the algebra\/
$\XpN$ are central.
\end{proposition}

\begin{proof}
We will work with the elements \eqref{3.22} where $n=3\,$. 
Using \eqref{3.3}~we~get
\begin{gather}
(R_{\ts13}(u\,,\ns v)\,R_{\ts23}(-\ts u\,,\ns v)\ot1)\,
T_1(u)\,T_2(-\ts u)\,T_3(v)=
\notag
\\[4pt]
T_3(v)\,T_1(u)\,T_2(-\ts u)\,
(R_{\ts13}(u\,,\ns v)\,R_{\ts 23}(-\ts u\,,\ns v)\ot1)\,.
\label{3.666666}
\end{gather}

Due to \eqref{rr} we have the identity in the algebra
$(\End\CN)^{\ot\ts3}(\ts u\,,\ns v\ts)$
\begin{equation*}
-\ts R_{\ts13}(u\,,-\ts v)\,P_{\ts12}\,R_{\ts23}(-\ts u\,,\ns v)=
-\ts P_{\ts12}\ns\left(1-\frac{1}{(u+v)^{\ts2}}\right).
\end{equation*}
By applying to it the antiautomorphism $\tau\,\pi$
relative to the first tensor factor of $(\End\CN)^{\ot\ts3}$
and then using \eqref{pt},\eqref{t1} we get
\begin{equation*}
Q_{\ts12}\,R_{\ts13}(u\,,\ns v)\,R_{\ts23}(-\ts u\,,\ns v)=
Q_{\ts12}\ns\left(1-\frac{1}{(u+v)^{\ts2}}\right).
\end{equation*}
Therefore by multiplying the relation
\eqref{3.666666} by $Q_{\ts12}\ot1$ on the left 
and using the above definition of the series $Z(u)$ we obtain
$$
(Q_{12}\ot Z(u))\ts\,T_3(v)\left(1-\frac{1}{(u+v)^{\ts2}}\right)=
T_3(v)\,\ts(Q_{12}\ot Z(u))\left(1-\frac{1}{(u+v)^{\ts2}}\right).
$$
The last relation implies that any generator $T_{ij}^{\ts(r)}$ of 
the algebra $\XpN$ commutes with all coefficients of the series $Z(u)\,$.
\qed
\end{proof}

By the definition of $Z(u)$ we have relations
in $(\End\CN)^{\ot\ts2}\ot\XpN\,[[u^{-1}]]$
\begin{align}
(Q\ot1)\,\ts T_1(u)\,T_2(-\ts u)&=Q\ot Z(u)\,,
\label{QTT}
\\[4pt]
T_2(-\ts u)\,T_1(u)\,(Q\ot1)&=Q\ot Z(u)\,.
\notag
\end{align}
Applying %the map
$\tau\,\pi\ot\id\ot\id$ to both sides of these relations
and then using \eqref{t1} gives
\begin{align*}
\TA_{\ts1}(u)\,(-\ts P\ot1)\,T_2(-\ts u)&=-\ts P\ot Z(u)\,,
\\[2pt]
T_2(-\ts u)\,(-\ts P\ot1)\,\TA_{\ts1}(u)&=-\ts P\ot Z(u)\,.
\end{align*}
Equivalently,
\vspace{-8pt}
\begin{align*}
\TA_1(u)\,T_1(-\ts u)\,(-\ts P\ot1)&=-\ts P\ot Z(u)\,,
\\[2pt]
(-\ts P\ot1)\,T_1(-\ts u)\,\TA_1(u)&=-\ts P\ot Z(u)\,.
\end{align*}
The latter two can be simpler written as relations in 
$(\End\CN)\ot\XpN\,[[u^{-1}]]$
\begin{align}
\label{TAT}
\TA(u)\,T(-\ts u)&=1\ot Z(u)\,,
\\[2pt]
\notag
T(-\ts u)\,\TA(u)&=1\ot Z(u)\,.
\end{align}
Observe that the last two
are also equivalent to each other due to Proposition~\ref{P3}. Explicitly,
these last two relations can be rewritten respectively as
\begin{align}
\label{Z1}
\sum_k\,
T_{\ts-k\ts,\ts-i}\ts (u)\,T_{\ts kj}(-\ts u)\,
{(-1)}^{\,\bj\ts\bk+\,\bj}
&=\de_{\ts ij}\,Z(u)\,,
\\[2pt]
\label{Z2}
\sum_k\,
T_{ik}(-u)\,T_{\ts-j\ts,\ts-k}\ts(u)\,
{(-1)}^{\,\bi\ts\bk+\,\bk}
&=\de_{\ts ij}\,Z(u)\,.
\end{align}

\begin{proposition}
\label{P4}
The antiautomorphism \eqref{S} %of $\XpN$ 
maps $Z(u)\mapsto Z(u)^{-1}\,$.
\end{proposition}

\begin{proof}
Due to Proposition \ref{P3} the relation \eqref{QTT} 
can be rewritten as
$$
Q\ot Z(u)^{-1}=(Q\ot1)\,\ts T_2(-\ts u)^{-1}\,T_1(u)^{-1}\,.
$$
The right hand side of the last displayed relation can also be obtained by
applying to the left hand side of \eqref{QTT} the antiautomorphism \eqref{S}
relative to the tensor factor $\XpN\,$.
Hence Proposition \ref{P4} follows from \eqref{QTT}.
\qed
\end{proof}

\begin{proposition}
\label{P5}
The square of antiautomorphism \eqref{S} of\/ $\XpN$ is given~by 
%the assignment
\begin{equation*}
T_{ij}(u)\mapsto Z(u)^{-1}\ts Z(-\ts u)\ts\,T_{ij}(u)\,.
\end{equation*}
\end{proposition}

\begin{proof}
By \eqref{TAT}
$$
T(u)^{-1}=Z(-\ts u)^{-1}\,\TA(-\ts u)\,.
$$
Thus for any indices $i$ and $j$ we have the relation
\begin{equation}
\label{TPij}
\TP_{ij}(u)=Z(-\ts u)^{-1}\,
T_{\,-j\ts,\ts-i}\ts(-\ts u)\,
{(-1)}^{\,\bi\ts\bj\ts+\,\bj}\ts.
\end{equation}
By applying the antiautomorphism \eqref{S} to the right hand side of
this relation, and then using the same relation for $-\ts j\ts,-\ts i$ 
and $-\ts u$ instead of $i\,,j$ and $u$ %respectively 
we get the series
$$
Z(u)^{-1}\,T_{ij}(u)\,Z(-\ts u)=
Z(u)^{-1}\ts Z(-\ts u)\ts\,T_{ij}(u)\,.
$$
Here we also used Propositions \ref{P3} and \ref{P4}.
\qed
\end{proof}

\newpage%%%%%%%%%%%%%%%%%%%%%%%%%%%%%%%%%%%%%%%%%%%%%%%%%%%%%%%%%%%%%%%%%%%%%%%

\begin{proposition}
\label{P6}
The automorphism \eqref{T} of\/ $\XpN$ maps $Z(u)\mapsto Z(-\ts u)\,$.
\end{proposition}

\begin{proof}
By setting $j=i$ in \eqref{Z1} we get an explicit formula
\begin{equation}
\label{CC1}
Z(u)=\sum_k\,
T_{\ts -k\ts,\ts-i}\ts (u)\,T_{\ts ki}(-\ts u)\,
{(-1)}^{\,\bi\ts\bk+\,\bi}\,.
\end{equation}
It follows from \eqref{CC1}
that the automorphism $\eqref{T}$ maps $Z(u)$ to the sum
$$
\sum_k\,
T_{\ts ik}\ts (u)\,T_{\ts-i,\ts-k}(-\ts u)\,
{(-1)}^{\,\bi\ts\bk+\,\bk}\,.
$$
On the other hand, by setting $j=i$ in \eqref{Z2} 
we get another explicit formula
\begin{equation}
\label{CC2}
Z(u)=\sum_k\,
T_{\ts ik}(-\ts u)\,T_{\ts-i\ts,\ts-k}\ts(u)\,
{(-1)}^{\,\bi\ts\bk+\,\bk}\,.
\end{equation}
Comparing the sum in \eqref{CC2} with the
previous display completes our proof.
\qed
\end{proof}

%------------------------------------------------------------------------------

\vspace{-\baselineskip}
\refstepcounter{section}
\subsection*{\bf\thesection}
\label{9}

There is a natural Hopf algebra structure on the extended Yangian $\XpN\,$. 
A coassociative comultiplication homomorphism $\De:\XpN\to\XpN\ot\XpN$ 
is defined by 
\begin{equation}
\label{3.7}
\De:\,T_{ij}(u)\,\mapsto\,\sum_{k}\,
T_{ik}(u)\ot T_{kj}(u)\,
{(-1)}^{\ts(\ts\bi\ts+\,\bk\ts)(\ts\bj\ts+\,\bk\ts)}
\end{equation}
where the tensor product is taken over the subalgebra $\CC[[u^{-1}]]$
of $\XpN\ts[[u^{-1}]]\,$.
The counit homomorphism $\ep:\XpN\to\CC$ is defined 
by $T_{ij}(u)\mapsto\de_{ij}\,$. Further,
the antipodal mapping $\S:\XpN\to\XpN$ is just 
the antiautomorphism \eqref{S}. Justification of these
definitions is very similar to that in the case 
of the Yangian of the general linear Lie algebra $\mathfrak{gl}_{\ts N}\ts$, 
see \cite[Section~4]{N4}. Here we omit the details. 

\begin{proposition}
\label{P7}
For the formal power series $Z(u)$ in $u^{-1}$ we have
\begin{equation}
\label{DC}
\De:\,Z(u)\mapsto Z(u)\ot Z(u)
\quad\text{and}\quad
\ep:\,Z(u)\mapsto1\,.
\end{equation}
\end{proposition}

\begin{proof}
By using the above formula 
\eqref{CC1} with the index $k$ replaced by $-\ts h\,$,
and then using the definition \eqref{3.7} twice
with $k$ replaced respectively by $-\ts k$ and $l\,$, 
the comultiplication homomorphism $\De$ maps $Z(u)$ to the sum
\begin{gather*}
\sum_{h,k,l}\,
(\,T_{\ts h,\ts -k}(u)\,T_{\ts-h,\ts l}(-\ts u)\ts)\ot
(\,T_{\ts -k,-i}(u)\,T_{\ts li}(-\ts u)\ts)\,
{(-1)}^{\,\bh\ts\bl\ts+\,\bi\ts\bk+\,\bk\ts\bl\ts+\,\bi\ts+\,\bk}
\\
=\ \sum_{k,l}\,\,(\,\de_{\ts kl}\,Z(u)\ts)\ot
(\,T_{\ts-k,-i}(u)\,T_{\ts li}(-\ts u)\ts)\,
{(-1)}^{\,\bi\ts\bk+\,\bi\ts+\,\bk\ts+\,\bl}
\\
=\ \sum_{k}\,\,Z(u)\ot
(\,T_{\ts-k,-i}(u)\,T_{\ts ki}(-\ts u)\ts)\,
{(-1)}^{\,\bi\ts\bk+\,\bi}=\,Z(u)\ot Z(u)
\end{gather*}
as needed. Here we also used the relation \eqref{Z1} with the indices
$k\ts ,l\ts,-\ts h$ instead of $i\ts,j\ts,k$ respectively.
Then we used the above formula \eqref{CC1} in its original form.

Due to the definition of the counit homomorphism $\ep\,$,
the second statement~in \eqref{DC} immediately follows from 
\eqref{CC1}.
\qed
\end{proof}

Note that 
Proposition \ref{P4} also follows from \eqref{DC},
because the multiplication %mapping
$\mu:\XpN\ot\XpN\to\XpN$ and the unit mapping
$\de:\CC\to\XpN:1\mapsto1$ obey %the general identity
$$
\mu\,(\ts\S\ot\,\id\ts)\,\De=%\mu\,(\id\ot\S)\,\De=
\de\,\ts\ep\,.
$$
Indeed, by applying to the coefficients of the series 
$Z(u)$ the homomorphisms at the two
sides of the last displayed 
identity, we get the equality 
$\S(Z(u))\,Z(u)=1\,$.

%------------------------------------------------------------------------------

\vspace{-\baselineskip}
\refstepcounter{section}
\subsection*{\bf\thesection}
\label{10}

The {\it Yangian} $\YpN$ of the Lie superalgebra $\pN$ 
is defined as the quotient of the $\ZZ_{\ts2}$-graded
algebra $\XpN$ by the relations 
\begin{equation*}
%\label{zz}
Z^{\ts(1)}=Z^{\ts(2)}=\ldots=0\,.
\end{equation*}
Hence $\YpN$ is a complex associative unital %$\ZZ_{\ts2}$-graded 
algebra with the generators \eqref{Tijr}, subject to the relations
\eqref{rttttr} and $Z(u)=1\,$. Here we still use the series \eqref{3.0}.
Explicit formulas for $Z(u)$ in terms of these series are given by
\eqref{CC1} and \eqref{CC2}.
The $\ZZ_{\ts2}$-grading on $\YpN$ is still defined by setting
$\deg T_{ij}^{\ts(r)}=\bi+\bj$ for every $r\,$.

Due Propositions \ref{P4} and \ref{P7}, the Hopf algebra structure
descends from $\XpN$ to the quotient $\YpN\,$. 
Due to Proposition \ref{P5}, the antipodal map on $\YpN$ is involutive.
In terms of the series \eqref{Tijr}, this map is still defined by 
\eqref{S}.

By $\eqref{TPij}$ the antipodal map on %the quotient
$\YpN$ coincides with the composition of the two commuting maps on $\YpN$
defined by \eqref{M} and \eqref{T}.
These two maps on $\YpN$ commute by Proposition \ref{P2}.
The map \eqref{M}
on $\XpN$ clearly descends to the quotient $\YpN\,$, whereas
\eqref{T} descends to $\YpN$ due to Proposition~\ref{P6}.

Consider the universal enveloping algebra $\UpN$ of the Lie superalgebra
$\pN\,$. This is an associative algebra generated by the elements 
$f_{\ts ij}\in\pN\,$, see \eqref{fij}.
The algebra $\UpN$ is $\ZZ_{\ts2}$-graded so that
$\deg f_{ij}=\bi+\bj\,$. The defining relations are %for these generators 
\begin{gather*}
f_{\,-j\ts,\ts-i}=-\,f_{\ts ij}\,{(-1)}^{\,\bi\ts\bj\ts+\,\bj}\ts,
\\
[\,f_{\ts ij}\,,f_{\ts kl}\,]=
\de_{\ts jk}\,f_{\ts il}-
\de_{\ts li}\,f_{\ts kj}\,
{(-1)}^{\ts(\,\bi+\,\bj\,)(\ts\bk+\,\bl\,)}
\\
-\,\de_{-\ts i,\ts k}\,f_{\ts-j,\ts l}\,{(-1)}^{\,\bi\ts\bj\ts+\,\bj}
\!-\de_{\ts j,-l}\,f_{\ts i,\ts -k}\,{(-1)}^{\,\bk\ts\bl\ts+\,\bl}\ts,
\end{gather*}
see \eqref{glN}. 
But now the square brackets stand for the supercommutator in $\UpN\,$.

On the other hand, by setting $Z(u)=1$ in \eqref{Z1} and taking 
the coefficients at $u^{-1}$ we obtain the relations in the algebra $\YpN$
$$
T_{\,-j\ts,\ts-i}^{\ts(1)}=
T_{\ts ij}^{\ts(1)}\,{(-1)}^{\,\bi\ts\bj\ts+\,\bj}\,.
$$
Multiplying \eqref{rttttr} by $u^{\ts2}-v^{\ts2}$ and then
taking the coefficients at $u\,v^{-1}$ yields
\begin{gather*}
[\,T_{\ts ij}^{\ts(1)}\,,T_{\ts kl}^{\ts(1)}\,]=
(\,\de_{\ts kj}\,T_{\ts il}^{\ts(1)}-\de_{\ts il}\,T_{\ts kj}^{\ts(1)}\,)\,
(-1)^{\ts\,\bi\ts\bk\ts+\,\bi\ts\bl\,+\,\bk\ts\bl}
\\
-\,\de_{-\ts i,\ts k}\,T_{\ts-j,\ts l}^{\ts(1)}\,
(-1)^{\ts\,\bi\bl\ts+\,\bj\ts\bl\ts+\,\bi}
\!+\de_{\ts j,-l}\,T_{\ts k,\ts -i}^{\ts(1)}\,
(-1)^{\,\bi\ts\bl\ts+\,\bi}\ts.
\end{gather*}
By comparing the last two displays with the above relations in $\UpN\,$,
we deduce that a homomorphism $\UpN\to\YpN$ can be defined by mapping
\begin{equation}
\label{UY}
-\ts f_{\ts ji}\,{(-1)}^{\,\bj}\mapsto T_{ij}^{\ts(1)}\,.
\end{equation}
Our Theorem \ref{T2} will imply that this homomorphism
is an embedding. Moreover it is  an embedding of Hopf algebras, because
for~$\XpN$ by our definitions 
\begin{equation*}
\De:\,T_{ij}^{\ts(1)}\mapsto T_{ij}^{\ts(1)}\ot1+1\ot T_{ij}^{\ts(1)}\ts,
\ \quad
\ep:\,T_{ij}^{\ts(1)}\mapsto0\,,
\ \quad
\S:\,T_{ij}^{\ts(1)}\mapsto-\,T_{ij}^{\ts(1)}\ts.
\end{equation*}

%------------------------------------------------------------------------------

%\vspace{-\baselineskip}
\refstepcounter{section}
\subsection*{\bf\thesection}
\label{11}

For any formal power series $c\ts(u)$ in $u^{-1}$
with coefficients from $\CC$ and leading term $1\ts$,
the automorphism \eqref{muf} of the associative algebra 
$\XpN$ maps $Z(u)$ to
$$
c\ts(u)\,c\ts(-u)\,Z(u)\,.
$$
%see \eqref{TAT}.
So %the automorphism 
\eqref{muf} %of $\XpN$ 
descends to %the quotient 
$\YpN$ if and only if $c\ts(u)\,c\ts(-u)=1\,$. 
The {\it special Yangian} of $\ts\pN$
is the fixed point subalgebra of $\YpN$ relative
to all these automorphisms. We will denote this subalgebra of $\YpN$
by $\ZpN\,$. It will not be confused with
the subalgebra of $\UpN$
consisting of all elements invariant relative to 
the adoint action of $\pN\,$,
since the latter subalgebra is just $\CC$ by \cite[Proposition~3]{S}.

The special Yangian $\ZpN$ is moreover a Hopf subalgebra of\/ $\YpN\,$.
Indeed, the anipodal map on %the Yangian 
$\YpN$ preserves the subalgebra $\ZpN$ by the definition of this subalgebra.
Now let $c\ts(u)$ by any formal power series in $u^{-1}$
with coefficients from $\CC$ and leading term $1\ts$, satisfying 
the condition $c\ts(u)\,c\ts(-u)=1\,$. 
By \eqref{3.7} the comultiplication on $\YpN$ maps
\begin{gather*}
c\ts(u)\,T_{ij}(u)\,\mapsto\,
\sum_{k}\,
(\ts c\ts(u)\,T_{ik}(u))\ot T_{kj}(u)\,
{(-1)}^{\ts(\ts\bi\ts+\,\bk\ts)(\ts\bj\ts+\,\bk\ts)}=
\\
\sum_{k}\,
T_{ik}(u)\ot(\ts c\ts(u)\, T_{kj}(u))\,
{(-1)}^{\ts(\ts\bi\ts+\,\bk\ts)(\ts\bj\ts+\,\bk\ts)}\,.
\end{gather*}
Hence the comultiplication intertwines the action
of the automorphism \eqref{muf} on $\YpN$ with its action on any one of the 
two tensor factors of $\YpN\ot\YpN\,$.
It follows that the image of the subalgebra $\ZpN$ of $\YpN$
under comultiplication is contained in both 
$\ZpN\ot\YpN$ and $\YpN\ot\ZpN\,$.
Therefore this image is contained in $\ZpN\ot\ZpN$ as needed.

We will conclude this section with an important observation. 
For any $t\in\CC$ consider the representation 
$\rho_{\,t}$ of $\XpN$ defined at the end of Section~\ref{6}.
By using the formula 
\eqref{TAT} and the relations
\eqref{t1},\eqref{rr} %and the definition of $\TA(u)\,$,
this representation maps $Z(u)$ to
$$
(\ts\tau\,\pi\ot\id\,)(R\ts(u\,,t\ts))\,R\ts(-\ts u\,,t\ts)=
R\ts(u\,,-\,t\ts)\,R\ts(-\ts u\,,t\ts)=
1-\frac{1}{(u+t\ts)^{\ts2}}\,.
$$
So the composition of $\rho_{\,t}$ with the automorphism 
\eqref{muf} of $\XpN$ maps $Z(u)$~to
$$
c\ts(u)\,c\ts(-u)
\left(1-\frac{1}{(u+t\ts)^{\ts2}}\right)
$$
which may be equal to $1$ for some $c\ts(u)$ if and only if $t=0\,$.
So the composition of $\rho_{\,t}$ with \eqref{muf}
may factor to a representation of %the quotient 
$\YpN$ if and only if $t=0\,$.
The choice of $c\ts(u)$ for that purpose is not unique then.
For example, we~can~choose 
$$
c\ts(u)=\frac{u}{u+1} 
\quad\text{or}\quad
c\ts(u)=\frac{u}{u-1}\,. 
$$

However, the composition of the tensor product 
$\rho_{\,t}\ot\rho_{\ts-\ts t}$ of representations of $\XpN$ with an automorphism \eqref{muf} may factor to a representation of %the quotient 
$\YpN$ for any $t\in\CC\,$.
This is because due to \eqref{DC}, the latter composition maps $Z(u)$~to
$$
c\ts(u)\,c\ts(-u)
\left(1-\frac{1}{(u+t\ts)^{\ts2}}\right)
\left(1-\frac{1}{(u-t\ts)^{\ts2}}\right).
$$
The choice of $c\ts(u)$ for $\rho_{\,t}\ot\rho_{\ts-\ts t}$
is still not unique. For example, we can choose
$$
c\ts(u)=\frac{(u+t\ts)^{\ts2}}{(u+t\ts)^{\ts2}-1} 
\quad\text{or}\quad
c\ts(u)=\frac{(u-t\ts)^{\ts2}}{(u-t\ts)^{\ts2}-1}\,. 
$$

%------------------------------------------------------------------------------

%\vspace{-\baselineskip}
\refstepcounter{section}
\subsection*{\bf\thesection}
\label{12}

There is a natural ascending $\ZZ\ts$-filtration 
on the associative algebra $\YpN\,$. It is defined by setting 
the degree of $T^{\ts(r)}_{ij}$ to $r$ 
for each $r\ge1$ and all indices $i\,,j\,$.
Let $\gr\YpN$ be the $\ZZ\ts$-graded algebra corresponding to this
filtration. Denote by $X^{\ts(r)}_{ij}$ the image of the generator
$T^{\ts(r)}_{ij}$ in the degree $r$ component of $\gr\YpN\,$.

The $\ZZ_2$-grading on $\YpN$ descends to $\gr\YpN$ so that
$\deg X^{\ts(r)}_{ij}=\bi+\ts\bj\,$.
For any $r,s\ge1$ by taking the coefficients at $u^{\ts-r}\ts v^{\ts-s}$ in 
\eqref{rttttr} we immediately obtain the supercommutation relation in the
$\ZZ_2$-graded algebra $\gr\YpN$
$$
[\ts\,X^{\ts(r)}_{\ts ij},\ts X^{\ts(s)}_{\ts kl}\,]=0\,.
$$
For any indices $i\,,j$ introduce the series with coefficients in 
the algebra $\gr\YpN$
\begin{equation*}
X_{\ts ij}(u)=
\de_{ij}\cdot1+
X_{\ts ij}^{\ts(1)}\ts u^{-1}+X_{\ts ij}^{\ts(2)}\ts u^{-2}+\ldots\,.
\end{equation*}
Then the relations \eqref{Z1} with $Z(u)=1$ imply 
that for all $i\,,j$ we have
\begin{equation}
\label{X1}
\sum_k\,
X_{\ts-k\ts,\ts-i\ts}\ts (u)\,X_{\ts kj\ts}(-\ts u)\,
{(-1)}^{\,\bj\ts\bk+\,\bj}=\de_{\ts ij}\,.
\end{equation}

The numerator in the first line of the defining 
relation \eqref{rttttr} of $\XpN$ vanishes at $v=u\,$.
Further, if we multiply by $u+v$ the expressions
in the second and in the third line of \eqref{rttttr}
and then set $v=-\,u\,$, we will get
$$
-\,\de_{\ts -i,k}\,\de_{\ts j,-l}\,Z(u)\,{(-1)}^{\,\bi\ts\bj}
\quad\text{and}\quad
\de_{\ts j,-l}\,\de_{\ts -i,k}\,Z(u)\,\,{(-1)}^{\,\bi\ts\bl+\,\bi}
$$ 
by using \eqref{Z1} and \eqref{Z2} respectively.
The latter two expressions cancel each other.
So by using \eqref{Z1} and \eqref{Z2} the defining relation
\eqref{rttttr} can be rewritten
is an equality of formal power series in $u^{-1},v^{-1}$
with coefficients in $\XpN\,$.

%The latter fact also follows from %the definition 
%\eqref{R} by using \eqref{QQ} and %the relation 
%\begin{equation*}
%\label{PP}
%(P\ot1)\,\ts T_1(u)\,T_2(u)=
%T_2(u)\,T_1(u)\,(P\ot1)\,.
%\end{equation*}

%The generators $X^{\ts(r)}_{ij}$ 
%of the supercommutative algebra $\gr\XpN$ are not free.
%For instance, by taking the coefficients at $u^{-1}$ in 
%\eqref{Z1} we obtain the relation
%$$
%X_{\,-j\ts,\ts-i}^{\ts(1)}=X^{\ts(1)}_{ij}
%\,{(-1)}^{\,\bi\ts\bj\ts+\,\bj}
%\quad\text{for}\quad
%i\neq j\,.
%$$

%==============================================================================

\vspace{-\baselineskip}
\refstepcounter{section}
\subsection*{\bf\thesection}
\label{13}

We will also employ another ascending $\ZZ\ts$-filtration on $\YpN\,$.
It is defined by setting 
the degree of $T^{\ts(r)}_{ij}$ to $r-1$ 
for any $r\ge1\,$.
The corresponding $\ZZ\ts$-graded algebra
will be denoted by $\grp\YpN\,$.
Let $Y^{\ts(r)}_{ij}$ be the image of
$T^{\ts(r)}_{ij}$ in the degree $r-1$~component of~$\grp\YpN\,$.

The $\ZZ_2$-grading on $\YpN$ descends to $\grp\YpN$ so that
$\deg Y^{\ts(r)}_{ij}=\bi+\ts\bj\,$.
For any $r\ge1$ by equating to $\de_{\ts ij}$
the left hand side of \eqref{Z1} %or that of \eqref{Z2},
and then taking the coefficients at $u^{\ts-r}$ 
we obtain the relation in the algebra $\grp\YpN$
\begin{equation}
\label{Y}
Y_{\ts-j\ts,\ts-i}^{\ts(r)}=-\,Y_{\ts ij}^{\ts(r)}\,
{(-1)}^{\,\bi\ts\bj+\,\bj\ts+\,r}\,.
\end{equation}

For $r\ts,\ns s\ge1$ by taking coefficients 
at $u^{\ts-r}\ts v^{\ts-s}$
in \eqref{rttttr} we obtain the relation %in $\grp\YpN$
\begin{gather}
[\,\,Y^{\ts(r)}_{\ts ij},\ts Y^{\ts(s)}_{\ts kl}\,]\,=
(\,\de_{\ts kj}\,Y^{\ts(r+s-1)}_{\ts il}-
\de_{\ts il}\,Y^{\ts(r+s-1)}_{\ts kj}\,)\,
(-1)^{\ts\,\bi\ts\bk\ts+\,\bi\ts\bl\,+\,\bk\ts\bl}\,+
\notag
\\
\label{YY}
\de_{-\ts i,\ts k}\,Y_{\ts-j,\ts l}^{\ts(r+s-1)}\,
(-1)^{\ts\,\bi\bl\ts+\,\bj\ts\bl\ts+\,\bi\ts+\,r}
-\,\de_{\ts j,-l}\,Y_{\ts k,\ts -i}^{\ts(r+s-1)}
\,{(-1)}^{\,\bi\ts\bl\ts+\,\bi\ts+\,r}\,.
\end{gather}
Here we used the equality to $1$ of the coefficient  
at $u^{\ts-r}\ts v^{\ts-s}$ in the expansion of 
$$
\frac{v^{\,1-r-s}-u^{\ts1-r-s}}{u-v}
$$
as a polynomial in $u^{-1},v^{-1}\,$. We also used the relation
\eqref{Y} obtained just above.

\newpage%%%%%%%%%%%%%%%%%%%%%%%%%%%%%%%%%%%%%%%%%%%%%%%%%%%%%%%%%%%%%%%%%%%%%%%

The structure of $\ZZ_2$-graded Hopf algebra
on $\YpN$ descends to $\grp\YpN\,$. Due to the definition \eqref{3.7} 
the comultiplication, 
the counit homomorphism and the 
antipodal antiautomorphism for $\grp\YpN$
are defined respectively by 
\begin{equation}
\label{gH}
Y_{ij}^{\ts(r)}\mapsto Y_{ij}^{\ts(r)}\ot1+1\ot Y_{ij}^{\ts(r)}\ts,
\ \quad
Y_{ij}^{\ts(r)}\mapsto0\,,
\ \quad
Y_{ij}^{\ts(r)}\mapsto-\,Y_{ij}^{\ts(r)}\ts.
\end{equation}

%==============================================================================

\vspace{-\baselineskip}
\refstepcounter{section}
\subsection*{\bf\thesection}
\label{14}

Now consider the twisted polynomial current Lie superalgebra $\g$ as
defined in Section \ref{1}. This subalgebra of $\glN[u]$ is spanned
by the elements
\begin{equation*}
\label{tij}
g^{\ts(r)}_{\ts ij}=e_{\ts ij}\,u^r\ns+\ts\om\ts(e_{\ts ij})\ts(-\ts u)^r=
e_{\ts ij}\,u^r\ns-e_{\ts-j\ts,\ts-i}\,u^r
{(-1)}^{\,\bi\ts\bj\ts+\,\bj\ts+\,r}
\end{equation*}
where $r=0,1,2,\,\ldots$ and $i\,,j=\pm\,1\lc\!\pm N\,$.
Here $g^{\ts(0)}_{\ts ij}=f_{\ts ij}$ by \eqref{fij}. For $r,s\ge0$
\begin{gather}
\label{t}
g^{\ts(r)}_{\ts-j\ts,\ts-i\ts}=-\ts\,g^{\ts(r)}_{\ts ij}
{(-1)}^{\,\bi\ts\bj\ts+\,\bj\ts+\,r}\,,
\\
\notag
[\,g_{\ts ij}^{\ts(r)}\,,g_{\ts kl}^{\ts(s)}\,]=
\de_{\ts jk}\,g_{\ts il}^{\ts(r+s)}-
\de_{\ts li}\,g_{\ts kj}^{\ts(r+s)}\,
{(-1)}^{\ts(\,\bi+\,\bj\,)(\ts\bk+\,\bl\,)}
\\
\label{tt}
-\,\de_{-\ts i,\ts k}\,g_{\ts-j,\ts l}^{\ts(r+s)}\,
{(-1)}^{\,\bi\ts\bj\ts+\,\bj\ts+\,r}
\!-\de_{\ts j,-l}\,g_{\ts i,\ts -k}^{\ts(r+s)}\,
{(-1)}^{\,\bk\ts\bl\ts+\,\bl\ts+\,s}\,.
\end{gather}

The universal enveloping algebra $\Ug$
is generated by the elements $g^{\ts(r)}_{\ts ij}$
as~an associative algebra. It is $\ZZ_2$-graded so that
the degrees of these generators are equal to $\bi+\ts\bj$ respectively.
The defining relations for these generators are \eqref{t} and
\eqref{tt} where the square brackets now denote the supercommutator.

By comparing \eqref{t},\eqref{tt} with 
\eqref{Y},\eqref{YY} above we deduce that the mapping
\begin{equation}
\label{get}
-\,g^{\ts(r)}_{\ts ji}\,(-1)^{\,\bj}
\mapsto
Y_{ij}^{\ts(r+1)}
\end{equation}
for $r\ge0$
defines a homomorphism $\Ug\to\grp\YpN$
of $\ZZ_2$-graded associative algebras. 
Moreover, this is a Hopf algebra homomorphism by
the definitions \eqref{gH}.

Note that for any\/ $t\in\CC$ the image of $T_{ij}^{\ts(r+1)}\in\YpN$
under the representation $\rho_{\,t}$
coincides with the image of the left hand side of \eqref{get}
under the {\it evaluation representation\/}
$
\operatorname{U}(\ts\glN[u]\ts)\mapsto\End\CN
$
defined by mapping $e_{\ts ij}\,u^r\mapsto E_{\ts ij}\,t^{\,r}\,$. 

%Also note that if $r=0$ then the left hand side of \eqref{get}
%coincides with that of~\eqref{UY}.

\begin{theorem}
\label{T2}
The mapping \eqref{get} defines an isomorphism 
$\Ug\to\grp\YpN$
of $\ZZ_2$-graded Hopf algebras. 
\end{theorem}

\begin{proof}
The homomorphism $\Ug\to\grp\YpN$ defined
by \eqref{get} above is clearly surjective. We have to prove that 
this homomorphism is injective as well.

Let $\IC$ be the set of all triples $(i,j,r)$ with $r\ge0\,$.
Choose any total ordering $<$ on this set. Also choose any
subset $\JC\subset\IC$ such that for $i\neq-\ts j$
it contains~only one of the triples $(i,j,r)$ and $(-\ts j\ts,-\ts i\ts,r)\,$,
while for $i=-\ts j$ it contains $(i,j,r)$
if and only if $\,\bi\ns+\ts r$ is odd.
Due to \eqref{t} the left hand sides of \eqref{get} 
corresponding to the triples $(i,j,r)\in\JC$ form a basis of the vector space 
$\g\,$. 

Let $\TC$ be the set of all %finite 
products of %generators 
$T_{ij}^{\ts(r+1)}$ %of $\YpN$
corresponding to $(i,j,r)\in\JC$ taken in non-decreasing orders,
such that two adjacent triples can be the same $(i,j,r)$ only if 
$\,\bi+\ts\bj=0$ in $\ZZ_2\,$.
We will prove that %the set 
$\TC$ is linearly independent~in~$\YpN\,$.
This is equivalent to the linear independence in $\grp\YpN$
of the set of products obtained from $\TC$ by replacing
every factor $T_{ij}^{\ts(r+1)}$ by the corresponding $Y_{ij}^{\ts(r+1)}\,$. 
Hence we will prove
the injectivity of the homomorphism defined by \eqref{get}.

Let $\SC$ be the set of all %finite 
products of the generators $T_{ij}^{\ts(r+1)}$ of the algebra $\YpN\,$.
Let us call a pair of triples $(i,j,r)$ and $(k,l,s)$
\emph{misordered} if $(i,j,r)>(k,l,s)\,$, or
$(i,j,r)=(k,l,s)$ and $\,\bi+\ts\bj=1$ in $\ZZ_2\,$.
For any product $A\in\SC$ let $\mu(A)$ be 
the number of pairs of factors of $A$
corresponding to misordered pairs of triples.
The pairs of factors of $A$ counted by $\mu\ts(A)$
need not to be adjacent. Let $\nu\ts(A)$ be the number of factors of $A$
corresponding to the triples which are not in $\JC\,$.

Let $A$ and $B$ be any two products from the set $\SC\,$.
Let $a$ and $b$ be the numbers of their factors respectively. 
Let $c$ and $d$ be respectively the sums of the degrees of their factors
relative to the $\ZZ\ts$-filtration on $\YpN$ from Section \ref{13}.
We will write $A\prec B$ if $c<d\,$, 
or $c=d$ but $a<b\,$, 
or $c=d$ and $a=b$ but $\nu\ts(A)<\nu\ts(B)\,$,
or $A$ is a permutation of factors of $B$ 
and $\mu\ts(A)<\mu\ts(B)\,$.
The relation denoted by %the symbol 
$\prec$ is a partial ordering on the set $\SC\ts$, with every descending chain
terminating. %Moreover it is a semigroup partial ordering.
Notice that the condition $A\prec B$ implies that $XA\,Y\prec XB\,Y$ for any
$X,Y\in\SC\,$.
 
Let $S(u)$ be the element of the algebra $(\End\CN)\ot\YpN\,[[u^{-1}]]$
inverse to $\TA(-\ts u)\,$. When determining the inverse element here
we do not use any relations in $\YpN\,$.
The relation \eqref{TAT} with $Z(u)=1$
can be now written as $S(u)=T(u)\,$. 
For any $i,j$ define a formal power series $S_{ij}(u)$ in $u^{-1}$
with coefficients in $\YpN\,$,
\begin{equation*}
S(u)=\sum_{i,j}\,E_{\ts ij}\ot S_{ij}(u)\,.
\vspace{-4pt}
\end{equation*}
Define the elements $S_{ij}^{\ts(1)},S_{ij}^{\ts(2)},\,\ldots\in\YpN$
as the coefficients of the series $S_{ij}(u)$ at $u^{-1},u^{-2},\,\ldots$
respectively. For any $r\ge0$
\begin{equation}
\label{red1}
S^{\ts(r+1)}_{\ts ij}=T_{\ts-j\ts,\ts-i}^{\ts(r+1)}\,
{(-1)}^{\,\bi\ts\bj+\,\bj\ts+\,r}+C
\end{equation}
where $C$ is a linear combination of
products of two or more generators of $\YpN$ with the sums of
the degrees of factors by the $\ZZ\ts$-filtration from Section \ref{13}
being less than $r\,$. These products occur in the linear combination 
$C$ only if $r>0\,$. 

Take any $(i,j,r)\notin\JC\,$. If $i\neq-\ts j\,$, then
$(-\ts j\ts,-\ts i\ts,r)\in\JC\,$. %by our choice of~$\JC\ts$.
In this case we will replace any factor $T^{\ts(r+1)}_{\ts ij}$ of %the product
$B$ by the right hand side of \eqref{red1}.
But if $i=-\ts j$ and $\,\bi\ns+\ts r$ is even, then 
we will replace %a factor 
$T^{\ts(r+1)}_{\ts ij}$ by $C/2\,$.
We will call either replacement a \emph{reduction\/} of $B\,$.
It does not change the value of $B\in\YpN$
%because $S^{\ts(r+1)}_{\ts ij}=T^{\ts(r+1)}_{\ts ij}$ in $\YpN$
and is {compatible} with $\prec$ on $\SC\,$, 
%in the sense of \cite[Section~1]{B}.
so that its result is a linear combination
of products~$A\prec B\,$. 

For any triples $(i,j,r)$ and $(k,l,s)$ the relation \eqref{rttttr}
gives an equality in $\YpN$
\begin{equation}
\label{red2}
T^{\ts(r+1)}_{\ts ij}\,T^{\ts(s+1)}_{\ts kl}=
T^{\ts(s+1)}_{\ts kl}\,T^{\ts(r+1)}_{\ts ij}\,
{(-1)}^{\ts(\,\bi+\,\bj\,)(\ts\bk+\,\bl\,)}+D
\end{equation}
where $D$ is a linear combination of single generators 
of $\YpN$ of degree $r+s$ by the $\ZZ\ts$-filtration from Section \ref{13},
and of products of two generators with the sums of the
degrees of factors being equal to $r+s-1\,$.
Here we use the expansions
$$
\frac1{u-v}\,=u^{-\ts1}+u^{-\ts2}\,v\,+\,\ldots
\quad\text{and}\quad
\frac1{u+v}\,=u^{-\ts1}-u^{-\ts2}\,v\,+\,\ldots
$$
then equate the coefficients at $u^{-\ts r-1}\ts v^{-s-1}$
of the series at the two sides of~\eqref{rttttr}. 

If $(i,j,r)>(k,l,s)$ then we will replace any product
of the corresponding~two adjacent factors of $B$
by the right hand side of \eqref{red2}.
But if $(i,j,r)=(k,l,s)$ and $\,\bi+\ts\bj=1$ in $\ZZ_2$ then
we will replace the product of the two factors of $B$ by $D/2\,$.
The value of $B\in\YpN$ will not change then.
Either replacement will be called a \emph{reduction\/} of $B$ too.
It is {compatible} with the ordering $\prec$ on %the set 
$\SC$ as well. 

\newpage%%%%%%%%%%%%%%%%%%%%%%%%%%%%%%%%%%%%%%%%%%%%%%%%%%%%%%%%%%%%%%%%%%%%%%%

When proving linear the independence of the set $\TC$ in $\YpN\,$,
we may assume that $(1,1,r)\notin\JC$ for every 
$r\ge0\,$. The collection of relations 
$S^{\ts(r+1)}_{\ts 11}=T^{\ts(r+1)}_{\ts 11}$ %in $\YpN$
can be written as the equality $S_{11}(u)=T_{11}(u)\,$.
Using only the relation \eqref{TAT} in 
$(\End\CN)\ot\YpN\ts[[u^{-1}]]\,$, 
this equality gives $Z(u)=1\,$. But the relation 
\eqref{TAT} follows from %the defining relation 
\eqref{3.3}, %for the extended Yangian $\XpN\,$,
see Section~\ref{8}. 
So the relation $Z(u)=1$ %for $\YpN$ 
follows from \eqref{3.3} and from the relations 
$S^{\ts(r+1)}_{\ts 11}=T^{\ts(r+1)}_{\ts 11}$ for every $r\ge0\,$.
Under our assumption on the choice of $\JC\ts$,
the latter relations are used to define
reductions via \eqref{red1}.

Let now us multiply each side of the relation \eqref{3.3}
by $P\ot1$ on the left and on the right. Also exchange the variables
$u$ and $v\,$. The result is the relation
\begin{equation}
\label{64}
(R\ts(-\ts u\,,-\ts v)\ot1)\,\ts T_2(v)\,T_1(u)=
T_1(u)\,T_2(v)\,(R\ts(-\ts u\,,-\ts v)\ot1)
\end{equation}
since by \eqref{pq}
\begin{equation}
\label{prp}
P\,R\ts(v\,,\ns u)\ot1)\,P=R\ts(-\ts u\,,-\ts v)\,.
\end{equation}
But \eqref{64} can also be obtained by multiplying each side
of \eqref{3.3} on the left and~on the right
by $R\ts(-\ts u\,,-\ts v)\ot1$ and then using \eqref{rr}. 
Inductively, this remark shows that the relations \eqref{red2} with 
$(i,j,r)\ge(k,l,s)$ imply all other relations~\eqref{red2}.
In the induction argument, we use the $\ZZ\ts$-filtration on $\YpN$
from Section \ref{12}.

If $(i,j)=(k,l)$ and $\,\bi+\ts\bj=0$ in $\ZZ_2\,$,
then the relation \eqref{rttttr} %for $\XpN$ 
simply means that the coefficients of the series $T_{ij}(u)$ pairwise commute.
Indeed, in this case the expressions in the second and the third lines
of \eqref{rttttr} vanish. Then \eqref{rttttr} becomes
$$
T_{ij}(u)\,T_{ij}(v)-T_{ij}(v)\,T_{ij}(u)=
\frac{T_{ij}(u)\,T_{ij}(v)-T_{ij}(v)\,T_{ij}(u)}{u-v}\,.
$$
Equivalently,
$$
T_{ij}(u)\,T_{ij}(v)=T_{ij}(v)\,T_{ij}(u)\,.
\hspace{2pt}
$$

In particular, we can eliminate from the definition of
$\YpN$ the relations \eqref{red2} with
$(i,j,r)=(k,l,s)$ and $\,\bi+\ts\bj=0$ in $\ZZ_2$
as tautological relations.
But to define reductions via \eqref{red2},
we used \eqref{red2} 
only when $(i,j,r)>(k,l,s)$ or
$(i,j,r)=(k,l,s)$ and $\,\bi+\ts\bj=1$~in~$\ZZ_2\,$.
The above arguments show that the used 
relations imply all other relations \eqref{red2}. 
Thus our reductions involve all defining relations of $\YpN\,$.

Let us now examine the \emph{inclusion ambiguities\/} of our set of
reductions. They correspond to misordered pairs of triples
where one or both triples are not in $\JC\ts$.
Let us show that they are resolvable relative to $\prec$
in the sense of \cite[Section~1]{B}.  

Take any product $T^{\ts(r+1)}_{\ts ij}\,T^{\ts(s+1)}_{\ts kl}$
where the pair of triples $(i,j,r)$ and $(k,l,s)$~is misordered.
Suppose that $(i,j,r)\notin\JC$ but $(k,l,s)\in\JC\,$.
Then we can apply to the above product
the reduction via \eqref{red1} in the first factor,
or alternatively the reduction via \eqref{red2}. 
The two results coincide by comparing \eqref{3.3} to the equality
\begin{equation}
\label{3.33}
(R\ts(u\,,\ns v)\ot1)\,\ts S_1(u)\,T_2(v)=
T_2(v)\,S_1(u)\,(R\ts(u\,,\ns v)\ot1)\,.
\end{equation}
Note that to prove the coincidence here we use
equalities of linear combinations of products
from $\SC$ which precede $T^{\ts(r+1)}_{\ts ij}\,T^{\ts(s+1)}_{\ts kl}$
relative to the partial ordering. 

To get the equality \eqref{3.33} we apply the map 
$\tau\,\pi\ot\id\ot\id$ to both sides of \eqref{3.3} and
then use \eqref{t1}. This yields the equality
$$
\TA_1(u)\,
(\ts R\ts(u\,,-\ts v)\ot1\ts)\,T_2(v)=
T_2(v)\,(\ts R\ts(u\,,-\ts v)\ot1\ts)\,\TA_1(u)\,.
$$
By multiplying both sides of this equality by
$\TA_1(u)^{-1}$ on the left and on the~right,
changing $u$ to $-\ts u$ and then using \eqref{rr} 
we indeed get \eqref{3.33}.

\newpage%%%%%%%%%%%%%%%%%%%%%%%%%%%%%%%%%%%%%%%%%%%%%%%%%%%%%%%%%%%%%%%%%%%%%%%

The case when $(i,j,r)\in\JC$ but $(k,l,s)\notin\JC$ can be treated 
in a similar way, but using instead of \eqref{3.33} the equality 
\begin{equation}
\label{3.333}
(R\ts(u\,,\ns v)\ot1)\,\ts T_1(u)\,S_2(v)=
S_2(v)\,T_1(u)\,(R\ts(u\,,\ns v)\ot1)\,.
\end{equation}
To get this equality
multiply each side of \eqref{3.3} by $R\ts(-\ts u\,,-\ts v)\ot1$
on the left and on the right, and then use \eqref{rr}. 
Apply the map $\id\ot\tau\,\pi\ot\id$ to both sides of the resulting 
equality and then use \eqref{t2}. In this way we obtain the equality 
$$
T_1(u)\,(\ts R\ts(u\,,-\ts v)\ot1\ts)\,\TA_2(v)=
\TA_2(v)\,(\ts R\ts(u\,,-\ts v)\ot1\ts)\,T_1(u)\,.
$$
Multiplying its both sides of this equality 
on the left and on the right by
$\TA_2(v)^{-1}$ and then changing $v$ to $-\ts v$ 
yields \eqref{3.333}.

The case when $(i,j,r)\notin\JC$ and $(k,l,s)\notin\JC$
can also be treated in a similar way, but using instead of \eqref{3.33}
the equality 
$$
(R\ts(u\,,\ns v)\ot1)\,\ts S_1(u)\,S_2(v)=
S_2(v)\,S_1(u)\,(R\ts(u\,,\ns v)\ot1)\,.
$$
To get this equality, multiply both sides of \eqref{tata}
on the left and on the right 
first by $\TA_1(u)^{-1}$ and then by $\TA_2(v)^{-1}\,$.
Then change $u\,,v$ to $-\ts u\,,\ns-\ts v$ respectively. 
 
Finally, let us examine the \emph{overlap ambiguities\/} of our set of
reductions. They correspond to the
sequences of triples of length three,
such that the pair of first two triples and
the pair of the last two triples in the sequence are misordered.
Let us show that they are also resolvable relative to $\prec$
in the sense of \cite[Section~1]{B}.  

Take any product 
$T^{\ts(r+1)}_{\ts ij}\,T^{\ts(s+1)}_{\ts kl}\,T^{\ts(t+1)}_{\ts gh}$
%from the set $\SC$
such that the pair of triples 
$(i,j,r)$ and $(k,l,s)$ is misordered,
and so is the pair of triples 
$(k,l,s)$ and $(g,h,t\ts)\,$. 
We can apply to the product the reduction via \eqref{red2}
in the first two factors. Alternatively, in the last two factors
we can apply the reduction via the equality 
obtained by replacing $(i,j,r)$ and $(k,l,s)$ in \eqref{red2}
respectively by $(k,l,s)$ and $(g,h,t\ts)\,$. 

To see that the results of two reductions coincide,
we can continue using \eqref{red2}. 
Thus for the first reduction we will use the equality
\begin{gather*}
(\ts R_{\ts23}(v\,,\ns w)\,R_{\ts13}(u\,,\ns w)\,R_{\ts12}(u\,,\ns v)\ot1\ts)
\,\ts T_1(u)\,T_2(v)\,T_3(w)=
\\[2pt]
T_3(w)\,T_2(v)\,T_1(u)\,\ts
(\ts R_{\ts23}(v\,,\ns w)\,R_{\ts13}(u\,,\ns w)\,R_{\ts12}(u\,,\ns v)\ot1\ts)
\end{gather*}
in the notation \eqref{3.22} with $n=3\,$. Here $w$ is another formal 
variable. In the same notation, for the second reduction
we will use the equality
\begin{gather*}
(\ts R_{\ts12}(u\,,\ns v)\,R_{\ts13}(u\,,\ns w)\,R_{\ts23}(v\,,\ns w)\ot1\ts)
\,\ts T_1(u)\,T_2(v)\,T_3(w)=
\\[2pt]
T_3(w)\,T_2(v)\,T_1(u)\,\ts
(\ts R_{\ts12}(u\,,\ns v)\,R_{\ts13}(u\,,\ns w)\,R_{\ts23}(v\,,\ns w)\ot1\ts)
\,.
\end{gather*}
The two results will coincide due to \eqref{rrr}.
Note that to prove the coincidence here we will use
equalities of linear combinations of products from $\SC$ which precede 
$T^{\ts(r+1)}_{\ts ij}\,T^{\ts(s+1)}_{\ts kl}\,T^{\ts(t+1)}_{\ts gh}$
relative to the partial ordering. 

So both the inclusion and the
overlap ambiguities of our set of reductions are resolvable.
But the elements of the set $\TC$ cannot be changed by the reductions.
The linear independence of $\TC$ in $\YpN$ is now ensured by 
\cite[Theorem~1.2]{B}. 
\qed
\end{proof}

\begin{corollary}
\label{C1}
The elements of\/ $\gr\YpN$ corresponding to\/ $T^{\ts(r+1)}_{\ts ij}$ 
with the triples $(i,j,r)\in\JC$ are free generators of this supercommutative algebra.
\end{corollary}

Our proof of Theorem \ref{T2} also implies the 
Poincar\'e\ts-Birkhoff\ts-\ns Witt Theorem 
for the algebra $\Ug\,$, see \cite[Theorem 5.15]{MM}.
Going back to our Section \ref{10} we obtain

\begin{corollary}
\label{C2}
The homomorphism $\UpN\to\YpN$ defined by the mapping \eqref{UY} 
is an embedding of $\ZZ_2$-graded Hopf algebras.
\end{corollary}

Let us link %the Yangian 
$\YpN$ to the Yangian $\YN$
of the general linear Lie algebra $\mathfrak{gl}_{\ts N}\,$.
The latter Yangian is a deformation of the universal enveloping algebra 
of the polynomial current Lie algebra $\mathfrak{gl}_{\ts N}[u]$ 
in the class of Hopf algebras.
Comparing \eqref{rttttr} with the relations in $\YN$ as given 
%for instance 
in \cite[Section~1]{N4} immediately yields
%obtain another corollary to Theorem \ref{T2}.

\begin{corollary}
\label{C3}
The associative subalgebra of\/ $\YpN$ generated 
by elements $T^{\ts(r)}_{\ts ij}$ with $i\,,j>0$ %and $r\ge1$ 
is isomorphic to the algebra $\YN\,$.
\end{corollary}

Note that the associative subalgebra of $\YpN$ 
appearing in Corollary \ref{C3} is not a Hopf subalgebra, 
see the definition \eqref{3.7}.
Thus the isomorphism of Corollary~\ref{C3} is not
an isomorphism of Hopf algebras.

%------------------------------------------------------------------------------

\vspace{-\baselineskip}
\refstepcounter{section}
\subsection*{\bf\thesection}
\label{15}

In this section we will explicitly describe the centre
of the Yangian $\YpN\,$.
An element of an $\ZZ_2\ts$-graded associative algebra
is {\it central\/} if its supercommutator
with every element of the algebra is zero.

First consider the subalgebra of the extended Yangian $\XpN$ generated by the elements $T^{\ts(r)}_{\ts ij}$ with $i\,,j>0\,$. 
This subalgebra of $\XpN$ is isomorphic to %the associative algebra 
$\YN\,$. Similarly to 
Corollary~\ref{C3} this isomorphism property follows from Theorem \ref{T2}
by %the defining relations
\eqref{rttttr}.
Consider the formal power series in $u^{-1}$
with coefficients in this subalgebra %of $\XpN$
\begin{equation}
\label{Au}
A\ts(u)=\sum_\si\,(-1)^{\ts\si}\,
T_{\,1,\si(1)}(u)\,
T_{\,2,\si(2)}(u+1)
\,\ldots\,
T_{\,N,\si(N)}(\ts u+N-1\ts)
\end{equation}
where $\si$ ranges over the symmetric group $\Sg_N$
permuting the indices $1\lc N\ts$. 
The power series in $(u+1)^{\ts-1}\lc(u+N-1)^{\ts-1}$ should be reexpanded
in $u^{-1}\,$. The series $A\ts(u)$ has the leading term $1\,$. 
The coefficients of $A\ts(u)$ at $u^{-1},u^{-2},\ts\ldots$ are free 
generators of the centre of that subalgebra of $\XpN\,$, see 
\cite[Section~2]{MNO}. All these coefficients 
have $\ZZ_2\ts$-degree zero,
hence the supercommutator with any of them in the algebra 
$\XpN$ is the usual commutator. 
By \cite[Proposition 2.7]{MNO} 
\begin{equation*}
A\ts(u)=\sum_\si\,(-1)^{\ts\si}\,
T_{\,\si(1),1}(\ts u+N-1\ts)\,\ldots\,
T_{\,\si(N-1),N-1}(\ts u+1\ts)\,
T_{\,\si(N),N}(u)\,.
\end{equation*}

Now take the series in $u^{-1}$ with 
coefficients in $\XpN$ and the leading term $1$
\begin{equation}
\label{66}
A\ts(u)\,A\ts(\ts1-N-u\ts)\,.
\end{equation}

\begin{theorem}
\label{T3}
The coefficients of the series \eqref{66} belong to the centre of\/ $\XpN\,$.
\end{theorem}

We will prove this theorem in the next nine sections. 
We will now employ it to describe the centre of the algebra $\YpN\,$.  
Let $B^{\ts(1)},B^{\ts(2)}\ldots$ be
the images in $\YpN$ of the coefficients of the series \eqref{66}
at $u^{-1},u^{-2},\ts\ldots$ respectively.
These images are central in $\YpN$ by Theorem \ref{T3}. 
They all are of $\ZZ_2\ts$-degree zero. 

\begin{corollary}
\label{C4}
The coefficients $B^{\ts(2)},B^{\ts(4)}\ldots$ 
freely generate the centre of\/ $\YpN\,$.
\end{corollary}

\begin{proof}
The elements $B^{\ts(2)},B^{\ts(4)}\ldots$ of $\YpN$ are central
by Theorem \ref{T3}. It
suffices to prove now that the images of these elements in the graded algebra
$\grp\YpN$ are free generators of its centre.
For $r\geqslant1$ the image of $B^{\ts(\ts2\ts r)}$ in $\grp\YpN$ equals
\begin{equation}
\label{B2r}
2\,Y_{11}^{\ts(\ts2\ts r)}+\ldots+2\,Y_{NN}^{\ts(\ts2\ts r)}\,.
\end{equation}

By Theorem \ref{T2} the algebra $\grp\YpN$
is isomorphic to the universal enveloping algebra 
$\Ug$ via \eqref{get}. Under this isomorphism the element \eqref{B2r} 
of $\grp\YpN$ with any $r\ge1$ corresponds to the element of $\Ug$
$$
-\,2\,g_{\ts11}^{\ts(\ts2\ts r-1)}-\ldots-2\,g_{\ts NN}^{\ts(\ts2\ts r-1)}=
-\,2\,(\ts e_{\ts 11}+e_{\ts-1\ts,\ts-1}+\ldots+
e_{\ts NN}+e_{\ts-N\ts,\ts-N}\ts)\,u^{\ts2\ts r-1}
$$
from our Section \ref{14}. The latter elements
freely generate of the centre of $\Ug\,$.

Indeed, they are algebraically independent
due to the Poincar\'e\ts-Birkhoff-Witt Theorem for 
$\Ug\,$, see \cite[Theorem 5.15]{MM}.
To show that they generate the centre of $\Ug$ consider 
the quotient of $\Ug$ by the ideal 
they generate. This quotient  
is isomorphic to the universal enveloping algebra $\Uh$ of the Lie superalgebra
$$
\h=\{\,h\ts(u)\in\a\ts[u]:\om\ts(\ts h\ts(u))=h\ts(-\ts u)\ts\}
$$
where $\a$ is the quotient of Lie superalgebra $\glN$ 
by the span of a central element
$$
e_{\ts 11}+e_{\ts-1\ts,\ts-1}+\ldots+e_{\ts NN}+e_{\ts-N\ts,\ts-N}\,.
$$
The automorphism $\om$ of $\glN$ maps this element to its negative
and therefore descends to the quotient $\a\,$.
But the centre of the Lie superalgebra $\a$ is trivial.
Hence the centre of %the algebra 
$\Uh$ is also trivial by \cite[Proposition 3.6]{N3}.
\qed
\end{proof}

To finish this section let us consider the images of the elements
$B^{\ts(1)},B^{\ts(2)}\ldots$ of $\YpN$
in the supercommutative algebra $\gr\YpN\,$.
In the notation of Section~\ref{12}
let $X(u)$ the $2\ts N\times 2\ts N$ matrix
whose $ij$ entry is the series 
$$
X_{ij}(u)\,{(-1)}^{\,\bi\ts\bj\ts+\,\bj}\,.
$$
We index the rows and columns of $X(u)$ by
the numbers $1\lc N,-\ts1\ts,\ldots,-\ts N\ts$.

The matrix $X(u)$ is invertible.
Let $\XP_{\ts ij}(u)$ be the $ij$ entry of the inverse matrix.
Then the relations \eqref{X1} imply that for any indices $i$ and $j$
$$
X_{\ts ij}(-\ts u)=\XP_{\ts-j\ts,\ts-i\ts}(u)\,.
$$
Therefore by using the definition 
\eqref{Au}\ts, the image in %the algebra 
$\gr\YpN\ts[[u^{-1}]]$ of the series 
\begin{equation*}
%\label{Bu}
B(u)=
1+B^{\ts(1)}\ts u^{-1}+B^{\ts(2)}\ts u^{-2}+\ldots
\end{equation*}
can be written as the product of the sum
$$
\sum_\si\,(-1)^{\ts\si}\,
X_{\ts1,\si(1)}(u)
\,\ldots\,
X_{\ts N,\si(N)}(u)
\vspace{-8pt}
$$
by the sum
\vspace{4pt}
$$
\sum_\si\,(-1)^{\ts\si}\,
\XP_{\ts-\si(1),-1}(u)
\,\ldots\,
\XP_{\ts-\si(N),-N}(u)\,.
$$
This product is just the {\it Berezinian\/} or the {\it superdeterminant\/}
of the matrix~$X(u)\,$. 
Hence the series $B(u)$ is an analogue for $\YpN$
of the quantum Berezinian \cite{N1} for the Yangian of 
the general linear Lie superalgebra $\glMN\,$, see \cite{N5} for details.

%------------------------------------------------------------------------------

%\vspace{-\baselineskip}
\refstepcounter{section}
\subsection*{\bf\thesection}
\label{16}

We will now begin our proof of Theorem \ref{T3}.
We will be use the same method as in the proof of Proposition \ref{P3}. 
However, the calculations will be more involved. 
Fix any $n\ge1$ and take the rational function of $u$ with values in 
$(\End\CN)^{\ts\ot\ts n}$
$$
H(u)=\prod_{1\le p<q\le n}\,R_{\,pq\ts}(\ts u+n-p\,,\ns u+n-q\ts)
$$
where the factors are arranged from left to right 
by ordering their pairs of indices $(p\,,q)$ lexicographically.
The ordering of factors
can be altered by employing~\eqref{rrr}, without affecting the function.
By using \eqref{3.3} repeatedly we get the equality 
$$
(\ts H(u)\ot1\ts)\ts\,T_1(u+n-1)\,\ldots\,T_n(u)=
T_n(u)\,\ldots\,T_1(u+n-1)\ts\,(\ts H(u)\ot1\ts)
$$  
of formal power series in $u^{-1}$ with coefficients in
$(\End\CN)^{\ts\ot\ts n}\ot\XpN\,$.

\begin{lemma}
\label{L1}
For any\/ $1\le p<q\le n$ we have the equalities in
$(\End\CN)^{\ts\ot\ts n}\ts(u)$
$$
P_{\ts pq}\ts H(u)=-\,H(u)
\quad\text{and}\quad
Q_{\ts pq}\ts H(u)=0\,.
$$
\end{lemma}

\begin{proof}
For $n=1$ we have nothing to prove. In particular, in this case $H(u)=1\,$.
Consider the case of $n=2\,$. Then $H(u)=R\ts(u+1\,,\ns u)$ while
$$
P\,R\ts(u+1\,,\ns u)=P\,\biggl(1-P+\frac{Q}{2\ts u+1}\,\biggr)=
P-1-\frac{Q}{2\ts u+1}=-\,R\ts(u+1\,,\ns u)
$$
by the equality $P^{\,2}=1$ and by the first equality in \eqref{pq}.
By the second and third equalities in \eqref{pq} we also have 
$$
Q\,R\ts(u+1\,,\ns u)=Q\,\biggl(1-P+\frac{Q}{2\ts u+1}\,\biggr)=
Q-Q-\frac{0}{2\ts u+1}=0\,.
$$

Now for any $n>2$ and for any given index $p<n$ let $q=p+1\,$. 
By \eqref{rrr} we can rearrange the factors of %the product 
$H(u)$ so that $R_{\,p,p+1\ts}(\ts u+n-p\,,\ns u+n-p-1\ts)$
is the leftmost factor. Then the above calculations for $n=2$ imply the
equalities of our lemma for any $n\ts$.
Hence our lemma holds whenever $q=p+1\,$.

In turn, this implies the required equalities for every $q>p\,$.
Indeed, the set of all elements $P_{\ts pq}$ and $Q_{\ts pq}$ 
with $p<q$ is generated in the algebra $(\End\CN)^{\ts\ot\ts n}$
by those elements where $q=p+1\,$.
\qed
\end{proof}

It follows from Lemma \ref{L1} that
every value of the normalized function $H(u)/\ts n\ts!$
is idempotent in the algebra $(\End\CN)^{\ts\ot\ts n}\,$. Indeed, 
by using the lemma we get
$$
H(u)\,H(u)=
\prod_{1\le p<q\le n}\biggl(\,1+\frac{1}{p-q}\,\biggr)\cdot
H(u)=n\ts!\,H(u)\,.
$$ 

Elements of %the algebra 
$(\End\CN)^{\ts\ot\ts n}$ act on %the vector space
$(\CN)^{\ts\ot\ts n}$ by \eqref{XY}~and~\eqref{xy}\ts .
Take any vector $e_{\ts i_1}\ot\ldots\ot e_{\ts i_n}$
with only positive indices $i_{\ts1}\lc i_{\ts n}\,$.
Then all elements $Q_{\ts pq}$ with $1\le p<q\le n$ annihilate this vector. 
So the image $H(u)\ts(\ts e_{\ts i_1}\ot\ldots\ot e_{\ts i_n})$ equals
$$
\prod_{1\le p<q\le n}\,
\biggl(\,1+\frac{P_{\ts pq}}{p-q}\,\biggr)
\cdot(\ts e_{\ts i_1}\ot\ldots\ot e_{\ts i_n})=\sum_\si\,(-1)^{\ts\si}\,
e_{\ts i_{\ts\si(1)}}\ot\ldots\ot e_{\ts i_{\ts\si(n)}}
$$
where $\si$ ranges over the group $\Sg_n\,$. The
last equality is well known \cite[Section~2]{MNO}. 

\begin{lemma}
\label{L2}
For any\/ $1\le p<n$ we have an equality in
$(\End\CN)^{\ts\ot\ts n}\ts(u)$
$$
H(u)\,Q_{\ts p,p+1}=
(\ts2\ts u+2\ts n-2\ts p-1\ts)\,H(u)\,(1+P_{\ts p,p+1})\,.
$$
\end{lemma}

\begin{proof}
For $n=1$ we have nothing to prove. 
Let $n=2\,$. Then $H(u)=R\ts(u+1\,,\ns u)$ while
by the relation \eqref{rr} we have
$$
R\ts(u+1\,,\ns u)\,R\ts(-\,u-1\,,-\,u)=0\,.
$$
Multiplying this equality by $2\ts u+1$ and using the definition 
of the second factor at the left hand side we get the equality
$$
R\ts(u+1\,,\ns u)\,(\ts (\ts2\ts u+1)\ts(\ts1+P\ts)-Q\ts)=0\,.
$$
This implies our lemma for $n=2$ and $p=1\,$. 
For any $n>2$ and $1\le p<n$ by using
\eqref{rrr} we can rearrange the factors of the product 
$H(u)$ so that the factor 
$$
R_{\,p,p+1\ts}(\ts u+n-p\,,\ns u+n-p-1\ts)
$$
is the rightmost. Then the above calculation with %the variable 
$u$ replaced by $u+n-p-1$ implies the lemma.
\qed
\end{proof}

%------------------------------------------------------------------------------

\vspace{-\baselineskip}
\refstepcounter{section}
\subsection*{\bf\thesection}
\label{17}

Let us continue our proof of Theorem \ref{T3}. 
For any $n\ge1$ consider the rational function of $u\,,v$ with values in 
the algebra $(\End\CN)^{\ts\ot\ts (n+1)}$
$$
F\ts(u\,,\ns v)=R_{\,1,n+1\ts}(\ts u+n-1\,,\ns v\ts)\,\ldots\, 
R_{\,n,n+1\ts}(\ts u\,,\ns v\ts)\,.
$$
%By using \eqref{rrr} repeatedly we get an equality 
%of rational functions of $u\,,\ns v$ 
%\begin{equation*}
%(\ts H(u)\ot1\ts)\,F\ts(u\,,\ns v)=
%R_{\,n,n+1\ts}(\ts u\,,\ns v\ts)\,\ldots\,
%R_{\,1,n+1\ts}(\ts u+n-1\,,\ns v\ts)\,
%(\ts H(u)\ot1\ts)\,.
%\end{equation*}
By using \eqref{3.3} repeatedly we get the equality 
of formal power series in $u^{-1},v^{-1}$
\begin{gather}
\notag
(\ts F\ts(u\,,\ns v)\ot1\ts)\ts\,
T_{\ts1}(u+n-1)\,\ldots\,T_{\ts n}(u)\,T_{\ts n+1}(v)=
\\[4pt]
\label{ftt}
T_{\ts n+1}(v)\,
T_{\ts1}(u+n-1)\,\ldots\,T_{\ts n}(u)\,\ts(\ts F\ts(u\,,\ns v)\ot1\ts)
\end{gather}
which have their coefficients in the algebra 
$(\End\CN)^{\ts\ot\ts(n+1)}\ot\XpN\,$.

\begin{lemma}
\label{L3}
For any $n\ge1$ the rational function
$(\ts H(u)\ot1\ts)\,F\ts(u\,,\ns v)$ of\/ $u\,,v$ with values in
$(\End\CN)^{\ts\ot\ts (n+1)}$
is equal to
$$
\biggl(\,1-\frac{P_{\,1,n+1}+\ldots+P_{\,n,n+1}}{u-v}+
\frac{Q_{\,1,n+1}+\ldots+Q_{\,n,n+1}}{u+v+n-1}\,\biggr)\,
(\ts H(u)\ot1\ts)
$$
plus terms divisible on the left by\/ 
$Q_{\ts pq}\in(\End\CN)^{\ts\ot\ts (n+1)}$ with 
$1\le p<q\le n\,$.
\end{lemma}

\begin{proof}
Using the relation \eqref{rrr} we can rewrite
the product $(\ts H(u)\ot1\ts)\,F\ts(u\,,\ns v)$ as
\begin{equation}
\label{induct}
R_{\,n,n+1\ts}(\ts u\,,\ns v\ts)\,\ldots\,R_{\,1,n+1\ts}(\ts u+n-1\,,\ns v\ts)
\,(\ts H(u)\ot1\ts)\,.
\end{equation}
We will prove by induction on $n$ that the latter product 
is equal to the sum in the lemma.
If $n=1$ then $H(u)=1$ while $F(u\,,\ns v)=R(u\,,\ns v)\,$. 
Then the required equality is just the definition of %the function 
$R(u\,,\ns v)$ as given in Section \ref{5}.

Suppose that $n>1\,$. By using the induction assumption and the definition 
of $H(u)$ for $n-1$ instead of $n\,$, the product \eqref{induct} is equal to 
\begin{gather*}
\biggl(\,1-\frac{P_{\,2,n+1}+\ldots+P_{\,n,n+1}}{u-v}+
\frac{Q_{\,2,n+1}+\ldots+Q_{\,n,n+1}}{u+v+n-2}\,\biggr)\,\times
\\[2pt]
\biggl(\,1-\frac{P_{\,1,n+1}}{u-v+n-1}+\frac{Q_{\,1,n+1}}{u+v+n-1}\,\biggr)\,
(\ts H(u)\ot1\ts)
\end{gather*}
plus terms divisible on the left by\/ 
$Q_{\ts pq}\in(\End\CN)^{\ts\ot\ts (n+1)}$ with 
$2\le p<q\le n\,$. We also used the relations \eqref{rrr}.
The product of two sums in the brackets equals
\begin{gather*}
1-\,\ts\frac{P_{\,1,n+1}}{u-v+n-1}+\frac{Q_{\,1,n+1}}{u+v+n-1}
\\[6pt]
-\,\ts\frac{P_{\,2,n+1}+\ldots+P_{\,n,n+1}}{u-v}+
\frac{Q_{\,2,n+1}+\ldots+Q_{\,n,n+1}}{u+v+n-2}\,\ts+
\\[6pt]
\frac{(\ts P_{\,2,n+1}+\ldots+P_{\,n,n+1}\ts)\,P_{\,1,n+1}}
{(\ts u-v\ts)\,(\ts u-v+n-1\ts)}
-
\frac{(\ts Q_{\,2,n+1}+\ldots+Q_{\,n,n+1}\ts)\,P_{\,1,n+1}}
{(\ts u+v+n-2\ts)\,(\ts u-v+n-1\ts)}
\\[6pt]
-\,\ts\frac{(\ts P_{\,2,n+1}+\ldots+P_{\,n,n+1}\ts)\,Q_{\,1,n+1}}
{(\ts u-v\ts)\,(\ts u+v+n-1\ts)}
+
\frac{(\ts Q_{\,2,n+1}+\ldots+Q_{\,n,n+1}\ts)\,Q_{\,1,n+1}}
{(\ts u+v+n-2\ts)\,(\ts u+v+n-1\ts)}\ .
\end{gather*}

Multiplying the first fraction in the third line by
$H(u)\ot1$ on the right gives
\begin{equation}
\label{72}
-\,\ts
\frac{(n-1)\,P_{\,1,n+1}}
{(\ts u-v\ts)\,(\ts u-v+n-1\ts)}
\end{equation}
times $H(u)\ot1\,$.
Here we used the first equality from Lemma \ref{L1} and the relation
$$
P_{\,p,n+1}\,P_{\,1,n+1}=P_{\,1,n+1}\,P_{\,1\ts p}
\quad\text{where}\quad
p=2\lc n\,.
$$
Multiplying the second fraction in the third line by
$H(u)\ot1$ on the right gives just zero due
to the second equality from Lemma \ref{L1}
and to the relations
$$
Q_{\,p,n+1}\,P_{\,1,n+1}=P_{\,1,n+1}\,Q_{\,p\ts1}=
-\,\ts P_{\,1,n+1}\,Q_{\,1\ts p}
\quad\text{where}\quad
p=2\lc n\,.
$$
The numerator of the first fraction in the fourth line can be written as
the sum~of
$$
P_{\,p,n+1}\,Q_{\,1,n+1}=Q_{\,1\ts p}\,P_{\,p,n+1}
\quad\text{over}\quad
p=2\lc n\,.
$$
These summands evidently have the divisibility property stated in the lemma.
Multiplying the second fraction in the fourth line by
$H(u)\ot1$ on the right gives
\begin{equation}
\label{73}
-\ts\,
\frac{Q_{\,2,n+1}+\ldots+Q_{\,n,n+1}}
{(\ts u+v+n-2\ts)\,(\ts u+v+n-1\ts)}
\end{equation}
times $H(u)\ot1\,$. 
Here we used the first equality from Lemma \ref{L1} and the relation
$$
Q_{\,p,n+1}\,Q_{\,1,n+1}=Q_{\,p,n+1}\,P_{\,1p}
\quad\text{where}\quad
p=2\lc n\,.
$$
By adding to the first and the second lines the expressions \eqref{72}
and \eqref{73} instead of the third and the 
fourth lines and by collecting similar terms we get the sum in brackets
displayed in the lemma.
Hence we make the induction step.
\qed
\end{proof}

%------------------------------------------------------------------------------

%\vspace{-\baselineskip}
\refstepcounter{section}
\subsection*{\bf\thesection}
\label{18}

Set $n=N\,$. Thus the above introduced
functions $H(u)$ and $F\ts(u\,,\ns v)$ will now take values in %the algebras
$(\End\CN)^{\ts\ot\ts N}$ and $(\End\CN)^{\ts\ot\ts(N+1)}$ respectively.
Elements of the latter algebra act on vectors in 
$(\CN)^{\ts\ot\ts(N+1)}$ by the conventions \eqref{XY}~and~\eqref{xy}.
From now on we will be denoting $z=u-v$ and $w=u+v+N-1$ for short.

\begin{lemma}
\label{L4}
Applying $(\ts H(u)\ot1\ts)\,F\ts(u\,,\ns v)$ to the vector
$e_{\ts 1}\ot\ldots\ot e_{\ts N}\ot e_{\ts1}$ has the same effect as
applying $H(u)\ot1$ to the vector
\begin{equation}
\label{v4}
\frac{z-1}z\,\,
e_{\ts 1}\ot\ldots\ot e_{\ts N}\ot e_{\ts1}\,.
\end{equation}
\end{lemma}

\begin{proof}
For $1\le p<q\le N+1$ the elements $Q_{\ts pq}\in(\End\CN)^{\ts\ot\ts(N+1)}$ 
annihilate the vector $e_{\ts 1}\ot\ldots\ot e_{\ts N}\ot e_{\ts1}\,$.
Hence applying $(\ts H(u)\ot1\ts)\,F\ts(u\,,\ns v)$ to this vector
has the same effect as applying
\begin{gather*}
\prod_{1\le p<q\le N}\,
\biggl(1+\frac{P_{\ts pq}}{p-q}\,\biggr)\cdot
\biggl(1-\frac{P_{\,1,N+1}}{u-v+N-1}\,\biggr)
\ts\ldots\ts
\biggl(1-\frac{P_{\,N,N+1}}{u-v}\,\biggr)=
\\
\prod_{1\le p<q\le N}\,
\biggl(1+\frac{P_{\ts pq}}{p-q}\,\biggr)\cdot
\biggl(\,1-\frac{P_{\,1,N+1}+\ldots+P_{\,N,N+1}}{u-v}\biggr)\,.
\end{gather*}
The last equality is well known, see for instance \cite[Section 2.3]{MNO}.
The summands $P_{\,2,N+1}\lc P_{\,N,N+1}$ of the last denominator
do not contribute to the image of the vector
$e_{\ts 1}\ot\ldots\ot e_{\ts N}\ot e_{\ts1}$ by the remark made just before stating Lemma~\ref{L2}. The summand $P_{\,1,N+1}$ does not change this vector.
Using the above mentioned remark once again we complete the proof of the lemma.
\qed
\end{proof}

\begin{lemma}
\label{L5}
Applying $(\ts H(u)\ot1\ts)\,F\ts(u\,,\ns v)$ to the vector
$e_{\ts 1}\ot\ldots\ot e_{\ts N}\ot e_{\ts-1}$ has~the same effect as
applying $H(u)\ot1$ to the linear combination
\begin{gather*}
\frac{w+1}w\,\,e_{\ts 1}\ot\ldots\ot e_{\ts N}\ot e_{\ts-1}
\\[0pt]
-\,\frac1z\,\,\sum_{k=1}^N\,\ts
e_{\ts1}\ot\ldots\ot e_{\ts k-1}
\ot e_{\ts-1}\ot 
e_{\ts k+1}\ot\ldots\ot e_{\ts N}\ot e_{\ts k}
\\[-5pt]
-\,\frac1w\,\,\sum_{k=1}^N\,\ts
e_{\ts-k}\ot e_{\ts2}\ot\ldots\ot e_{\ts N}\ot e_{\ts k}\,.
\end{gather*}
\end{lemma}

\begin{proof}
Applying $F\ts(u\,,\ns v)$ to %the vector 
$e_{\ts 1}\ot\ldots\ot e_{\ts N}\ot e_{\ts-1}$ has the same effect as
applying
$$
R_{\,1,N+1\ts}(\ts u+N-1\,,\ns v\ts)\,
\biggl(\,1-\frac{P_{\,2,N+1}}{u-v+N-2}\,\biggr)
\ts\ldots\ts
\biggl(\,1-\frac{P_{\,N,N+1}}{u-v}\biggr)\,.
$$
The remark made before stating Lemma~\ref{L2}
remains valid if $n=N$ and one of the indices $i_{\ts1}\lc i_{\ts N}$
is $-\ts1$ while the other $N-1$ indices are $2\lc N$ in any order.
Hence when applying $(\ts H(u)\ot1\ts)\,F\ts(u\,,\ns v)$
to our vector we can %further 
replace %the factor 
$F\ts(u\,,\ns v)$~by
\begin{equation*}
R_{\,1,N+1\ts}(\ts u+N-1\,,\ns v\ts)\,
\biggl(\,1-\frac{P_{\,2,N+1}+\ldots+P_{\,N,N+1}}{u-v}\,\biggr)\,.
\end{equation*}
Here we also use Lemma \ref{L2} and the above mentioned result from
\cite[Section~2.3]{MNO}.

By simply opening the brackets the latter product equals
\begin{gather*}
1-\frac{P_{\,1,N+1}}{u-v+N-1}+\frac{Q_{\,1,N+1}}{u+v+N-1}
-\frac{\,P_{\,2,N+1}+\ldots+P_{\,N,N+1}}{u-v}\,\ts+
\\[6pt]
\frac{\,P_{\,1,N+1}\,(\,P_{\,2,N+1}+\ldots+P_{\,N,N+1}\ts)}
{(\ts u-v+N-1\ts)\,(\ts u-v\ts)}-
\frac{\,Q_{\,1,N+1}\,(\,P_{\,2,N+1}+\ldots+P_{\,N,N+1}\ts)}
{(\ts u+v+N-1\ts)\,(\ts u-v\ts)}\,.
\end{gather*}
The last two denominators are respectively equal to the sums
$$
P_{\,12}\,P_{\,1,N+1}+\ldots+P_{\,1N}\,P_{\,1,N+1}
\quad\text{and}\quad
P_{\,2,N+1}\,Q_{\,12}+\ldots+P_{\,N,N+1}\,Q_{\,1N}\,.
$$
The second of the two sums annihilates our vector.
When applying the first sum to our vector we can replace 
each of the factors $P_{\,12}\lc P_{\,1N}$ just by $-\ts 1$.
Here we again use Lemma \ref{L2}.
So when applying $(\ts H(u)\ot1\ts)\,F\ts(u\,,\ns v)$
to our vector we can replace the function $F\ts(u\,,\ns v)$
by the sum
\begin{gather*}
1-\frac{P_{\,1,N+1}}{u-v+N-1}+\frac{Q_{\,1,N+1}}{u+v+N-1}
-\frac{P_{\,2,N+1}+\ldots+P_{\,N,N+1}}{u-v}
\\[5pt]
-\,\frac{(N-1)\,P_{\,1,N+1}}
{(\ts u-v+N-1\ts)\,(\ts u-v\ts)}=
\\[5pt]
1-\frac{P_{\,1,N+1}+P_{\,2,N+1}+\ldots+P_{\,N,N+1}}{z}
+\frac{Q_{\,1,N+1}}{w}\ts\,.
\end{gather*}

Applying the expression in the last displayed line to 
$e_{\ts 1}\ot\ldots\ot e_{\ts N}\ot e_{\ts-1}$ gives
\begin{gather*}
e_{\ts 1}\ot\ldots\ot e_{\ts N}\ot e_{\ts-1}
-\,\frac1z\,\,\sum_{k=1}^N\,\ts
e_{\ts1}\ot\ldots\ot e_{\ts k-1}
\ot e_{\ts-1}\ot 
e_{\ts k+1}\ot\ldots\ot e_{\ts N}\ot e_{\ts k}
\\[-5pt]
-\,\frac1w\,\,\sum_{k=1}^N\,\ts
e_{\ts-k}\ot e_{\ts2}\ot\ldots\ot e_{\ts N}\ot e_{\ts k}
+\,\frac1w\,\,\sum_{k=1}^N\,\ts
e_{\ts k}\ot e_{\ts2}\ot\ldots\ot e_{\ts N}\ot e_{\ts-\ts k}\,.
\end{gather*}
By Lemma \ref{L2} applying $H(u)\ot1\ts$
to the last sum over $k=1\lc N$ annihilates all summands 
but one where $k=1\,$.
This observation completes the proof. 
\qed
\end{proof}

%------------------------------------------------------------------------------

\vspace{-\baselineskip}
\refstepcounter{section}
\subsection*{\bf\thesection}
\label{19}

Introduce the rational function of $u\,,v$ with values in 
$(\End\CN)^{\ts\ot\ts (2\ts N+1)}$
$$
C\ts(u\,,\ns v)=F^{\,\prime}\ts(u\,,\ns v)\,
F^{\,\prime\prime}\ts(\ts1-u-N\,,\ns v)
$$
where $F^{\,\prime}\ts(u\,,\ns v)$ and 
$F^{\,\prime\prime}\ts(u\,,\ns v)$ are the images of $F\ts(u\,,\ns v)$
by the two embeddings
$$
(\End\CN)^{\ts\ot\ts (N+1)}\to(\End\CN)^{\ts\ot\ts (2\ts N+1)}
$$
defined as 
$\iota_{\,1}\ns\ot\ldots\ot\ts\iota_{\,N}\ot\ts\iota_{\,2\ts N+1}$ 
and as
$\iota_{\,N+1}\ns\ot\ldots\ot\ts\iota_{\,2\ts N}\ot\ts\iota_{\,2\ts N+1}$
respectively.
Using %the relation 
\eqref{ftt} twice we obtain that multiplying the formal power series 
in $u^{-1},v^{-1}$
$$
T_{\ts1}(u+N-1)\,\ldots\,T_{\ts N}(u)\,
T_{\ts N+1}(-\ts u)\,\ldots\,T_{\,2\ts N}(\ts1-N-u\ts)\,
T_{\,2\ts N+1}(v)
$$
by $C\ts(u\,,\ns v)\ot1$ on the left gives the same result as  
as multuplying by $C\ts(u\,,\ns v)\ot1$ on the right the series
$$
T_{\,2\ts N+1}(v)\,
T_{\ts1}(u+N-1)\,\ldots\,T_{\ts N}(u)\,
T_{\ts N+1}(-\ts u)\,\ldots\,T_{\,2\ts N}(\ts1-N-u\ts)\,.
$$
Both series have their coefficients in the algebra 
$(\End\CN)^{\ts\ot\ts(2\ts N+1)}\ot\XpN\,$.

\newpage%%%%%%%%%%%%%%%%%%%%%%%%%%%%%%%%%%%%%%%%%%%%%%%%%%%%%%%%%%%%%%%%%%%%%%

For any index $j=\pm\,1\lc\!\pm\ns N$ denote
\begin{equation}
\label{fj}
f_{\ts j}=
e_{\ts 1}\ot\ldots\ot e_{\ts N}\ot e_{\ts 1}\ot\ldots\ot e_{\ts N}\ot 
e_{\ts j}\,.
\end{equation}
Observe that replacing $u$ by $1-N-u$ maps $z=u-v$
to $1-N-u-v=-\,w\,$. Therefore by using Lemma \ref{L4} 
with that replacement of the variable $u$ and then~in 
the original formulation we immeditaly get the following corollary
to this lemma.

\begin{corollary}
\label{C5}
Applying $(\ts H(u)\ot H(\ts1-N-u\ts)\ot1\ts)\,C\ts(u\,,\ns v)$ to the vector
$f_{\ts 1}$ has the same effect 
as applying $H(u)\ot H(\ts1-N-u\ts)\ot1$ to the vector
\begin{equation}
\label{v5}
\frac{(z-1)\,(w+1)}{z\,w}\,f_{\ts 1}\,.
\end{equation}
\end{corollary}

Below is the counterpart of Corollary \ref{C5}
for our Lemma \ref{L5} instead of Lemma~\ref{L4}.

\begin{corollary}
\label{C6}
Applying $(\ts H(u)\ot H(\ts1-N-u\ts)\ot1\ts)\,C\ts(u\,,\ns v)$ to %the vector
$f_{\ts-1}$ has the same effect 
as applying $H(u)\ot H(\ts1-N-u\ts)\ot1$ to the linear combination
\begin{gather}
\notag
\hspace{-10pt}
\frac{(z-1)\,(w+1)}{z\,w}\,f_{\ts-1}
\\[0pt]
\notag
\hspace{-10pt}
-\,\ts\frac{z-1}{z^{\ts2}}\,\,\sum_{k=1}^N\,\ts
e_{\ts1}\ot\ldots\ot e_{\ts k-1}
\ot e_{\ts-1}\ot 
e_{\ts k+1}\ot\ldots\ot e_{\ts N}\ot 
e_{\ts 1}\ot\ldots\ot e_{\ts N}\ot e_{\ts k}
\\[-5pt]
\notag
\hspace{-10pt}
\qquad\qquad
-\,\ts\frac{z-1}{z\,w}\,\,\sum_{k=1}^N\,\ts
e_{\ts-k}\ot e_{\ts2}\ot\ldots\ot e_{\ts N}\ot 
e_{\ts 1}\ot\ldots\ot e_{\ts N}\ot e_{\ts k}\,\ts+
\\[-5pt]
\notag
\hspace{10pt}
\frac{z-1}{z\,w}\,\,\sum_{k=1}^N\,\ts
e_{\ts 1}\ot\ldots\ot e_{\ts N}\ot
e_{\ts1}\ot\ldots\ot e_{\ts k-1}
\ot e_{\ts-1}\ot 
e_{\ts k+1}\ot\ldots\ot e_{\ts N}\ot e_{\ts k}\,\ts+
\\[-5pt]
\label{fivel}
\hspace{36pt}
\frac{z-1}{z^{\ts2}}\,\,\sum_{k=1}^N\,\ts
e_{\ts 1}\ot\ldots\ot e_{\ts N}\ot
e_{\ts-k}\ot e_{\ts2}\ot\ldots\ot e_{\ts N}\ot e_{\ts k}\,.
\end{gather}
\end{corollary}

\begin{proof}
By Lemma \ref{L5} with the variable 
$u$ replaced by $1-N-u$ we immediately obtain that applying
$(\ts 1\ot H(\ts1-N-u\ts)\ot1\ts)\,F^{\,\prime\prime}\ts(\ts1-u-N\,,\ns v)$ 
to the vector $f_{\ts-1}$ has the same effect 
as applying $1\ot H(\ts1-N-u\ts)\ot1$
to the linear combination 
\begin{gather}
\notag
\frac{z-1}z\,f_{\ts-1}\,\ts+
\\[0pt]
\notag
\frac1w\,\,\sum_{k=1}^N\,\ts
e_{\ts 1}\ot\ldots\ot e_{\ts N}\ot
e_{\ts1}\ot\ldots\ot e_{\ts k-1}
\ot e_{\ts-1}\ot 
e_{\ts k+1}\ot\ldots\ot e_{\ts N}\ot e_{\ts k}\,\ts+
\\[-5pt]
\label{threelines}
\frac1z\,\,\sum_{k=1}^N\,\ts
e_{\ts 1}\ot\ldots\ot e_{\ts N}\ot
e_{\ts-k}\ot e_{\ts2}\ot\ldots\ot e_{\ts N}\ot e_{\ts k}\,.
\end{gather}

Further, 
by the original Lemma \ref{L5} applying
$(\ts H(u)\ot1\ot1\ts)\,F^{\,\prime}\ts(u\,,\ns v)$
to the first of three lines \eqref{threelines}
has the same effect as applying $H(u)\ot1\ot1$
to the sum of first three lines in \eqref{fivel}.
Further, applying $(\ts H(u)\ot1\ot1\ts)\,F^{\,\prime}\ts(u\,,\ns v)$
to the second and to the third lines of \eqref{threelines}
has the same effect as applying the product
\begin{equation*}
(\ts H(u)\ot1\ot1\ts)\,
\biggl(\,1-\frac{P_{\,1,2\ts N+1}}{u-v+N-1}\,\biggr)
\ts\ldots\ts
\biggl(\,1-\frac{P_{\,N,2\ts N+1}}{u-v}\biggr)\,.
\end{equation*}

Applying the last displayed product
to the second and third lines of \eqref{threelines} gives
the same result as applying the product
\begin{equation}
\label{numerat}
(\ts H(u)\ot1\ot1\ts)\,
\biggl(
\,1-\frac{P_{\,1,2\ts N+1}+\ldots+P_{\,N,2\ts N+1}}{u-v}
\biggr)\,.
\end{equation}
Here we used Lemma \ref{L2} like in the beginning of our proof of Lemma \ref{L5}.
Let us~now use the remark made just before stating Lemma \ref{L2}.
When applying \eqref{numerat}
to the summands with the index $k$ in the second and third lines of
\eqref{threelines}, 
all summands of the numerator in \eqref{numerat}
vanish except for $P_{\,k\ts,\ts2\ts N+1}\,$.
The latter does not change the summands with the index $k$
in the second and third lines of
\eqref{threelines}. So applying 
$(\ts H(u)\ot1\ot1\ts)\,F^{\,\prime}\ts(u\,,\ns v)$
to the second and the third lines of \eqref{threelines}
gives the same as applying $H(u)\ot1\ot1$ respectively
to the fourth and the fifth lines of \eqref{fivel}.~\qed
\end{proof}

%------------------------------------------------------------------------------

\vspace{-\baselineskip}
\refstepcounter{section}
\subsection*{\bf\thesection}
\label{20}
So far we used the equality in Lemma \ref{L2} only in the %special 
cases when both sides~of the equality annihilate a given vector. 
In the next section we will use that equality in more general cases.
We will employ the following corollary to Lemma~\ref{L2}. Denote
\begin{equation*}
%label{dk}
d_{\ts k}=
e_{\ts1}\ot\ldots\ot e_{\ts k-1}\ot e_{\ts k+1}\ot\ldots\ot e_{\ts N}
\quad\text{for}\quad
k=1\lc N\ts.
\end{equation*}

\begin{corollary}
\label{C7}
For any $k>1$ applying $H(u)$ to the vector\/
$e_{\ts-k}\ot e_{\ts2}\ot\ldots\ot e_{\ts N}$
has the same effect as applying $H(u)$ to the linear combination
\begin{gather*}
a(u)\,d_{\ts k}\ot e_{\ts-1}\,(-1)^{\ts N+k}-\ts
b\ts(u)\,d_{\ts 1}\ot e_{\ts-k}\,(-1)^{\ts N}
\end{gather*}
where
$$
a(u)=\frac{2}{\ts2u+N+1\ts}
\quad\text{and}\quad
b\ts(u)=\frac{2u+N-1}{\ts2u+N+1\ts}\,\,.
$$
Similarly, 
for any index $k>1$ applying $H(\ts1-N-u\ts)$ to the vector
$$
e_{\ts1}\ot\ldots\ot e_{\ts k-1}\ot e_{\ts-1}\ot 
e_{\ts k+1}\ot\ldots\ot e_{\ts N}
$$
has the same effect as applying $H(\ts1-N-u\ts)$ to the linear combination
\begin{gather*}
a(u)\,e_{\ts-k}\ot d_{\ts 1}-
b\ts(u)\,e_{\ts-1}\ot d_{\ts k}\,(-1)^{\ts k}\,.
\end{gather*}
\end{corollary}

\begin{proof}
%Fix any index $k>1\,$. 
For any index $i$ denote by $x_{\ts i}(u)$ the result
of applying $H(u)$ to the vector 
\begin{equation}
\label{xvector}
e_{\ts2}\ot\ldots\ot e_{\ts k-1}\ot e_{\ts k+1}\ot\ldots\ot e_{\ts N}
\ot e_{\ts i}\ot e_{\ts -\ts i}\,.
\end{equation}
By Lemma \ref{L2} applying $H(u)$ to the vector
$e_{\ts-k}\ot e_{\ts2}\ot\ldots\ot e_{\ts N}$
gives the vector $(-1)^{\ts k}\,x_{\ts-\ts k}(u)\,$.
Let us apply both sides of the equality in Lemma \ref{L2}
with $n=N$ and $p=N-1$ to the vectors \eqref{xvector} with $i=1\lc N\ts$. 
We obtain the equations
\begin{equation}
\label{lineq}
x_{\ts 1}(u)-x_{\ts-1}(u)+\ldots+x_{\ts N}(u)-x_{\ts-\ts N}(u)=
(2u+1)\,(\,x_{\ts i}(u)+x_{\ts -\ts i}(u)\ts)\,.
\end{equation}
But for $i=2\lc k-1\ts,k+1\ts\lc N$ we have
$x_{\ts i}(u)=0$ by Lemma~\ref{L2}.
So by \eqref{lineq} the vectors $x_{\ts -\ts i}(u)$
with $i=2\lc k-1\ts,k+1\ts\lc N$ are equal to each~other.
Using these two observations, it is easy to deduce from the
equations \eqref{lineq} that 
$$
x_{\ts-k}(u)=
a(u)\,x_{\ts1}(u)-b\ts(u)\,x_{\ts k}(u)\,.
$$
By using Lemma \ref{L2} again we now obtain the first statement of
our corollary.
 
Changing the notation, %for $i=\pm\,1\lc\pm N$ 
let $x_{\ts i}(u)$ be
the result
of applying $H(\ts1-N-u\ts)$ to %the vector 
\begin{equation}
\label{yvector}
e_{\ts -\ts i}\ot e_{\ts i}\ot e_{\ts 2}\ot\ldots\ot e_{\ts k-1}\ot 
e_{\ts k+1}\ot\ldots\ot e_{\ts N}\,.
\end{equation}
By Lemma \ref{L2} applying $H(\ts1-N-u\ts)$ to\/ %the vector
$
e_{\ts1}\ot\ldots\ot e_{\ts k-1}\ot e_{\ts-1}\ot 
e_{\ts k+1}\ot\ldots\ot e_{\ts N}
$
gives the vector $(-1)^{\ts k}\,x_{\ts-\ts 1}(u)\,$. 
Let us replace %the variable 
$u$ by $1-N-u$ in the equality in Lemma~\ref{L2}
and apply both sides of the resulting equality with $n=N$ and $p=1$ 
to the vectors \eqref{yvector} with $i=1\lc N\,$.
In this way we obtain the same equations \eqref{lineq}
although for the changed vectors.
For $i=2\lc k-1\ts,k+1\ts\lc N$ we still have
$x_{\ts i}(u)=0$ by Lemma~\ref{L2}. So
the equations \eqref{lineq} easily imply that 
$$
x_{\ts-1}(u)=
a(u)\,x_{\ts k}(u)-b\ts(u)\,x_{\ts 1}(u)\,.
$$
By using Lemma \ref{L2} again we now get the second statement of
the corollary.
\qed
\end{proof}

%------------------------------------------------------------------------------

\vspace{-\baselineskip}
\refstepcounter{section}
\subsection*{\bf\thesection}
\label{21}

Let us relate two vectors in $(\CN)^{\ts\ot\ts(2\ts N+1)}(u\,,v)$
by the symbol $\sim$ to each other if 
their difference vanishes under the action of
$H(u)\ot H(1-N-u)\ot1\,$, 
or lies in the image of the action of $Q_{\ts N,\ts N+1}\,$.
Extend the relation $\sim$ transitively.
Recall the notations $z=u-v$ and $w=u+v+N-1$ adopted in Section \ref{18}.

\begin{proposition}
\label{P8}
For any index\/ $j=\pm\,1\lc\!\pm\ns N$ we have the relation
$$
C\ts(u\,,\ns v)\,f_{\ts j}\,\sim\,\frac{(z-1)\,(w+1)}{z\,w}\,f_{\ts j}\,.
$$
\end{proposition}

\begin{proof}
%First observe that 
Proposition \ref{P8} follows from its
particular cases of $j=1\,,-\ts1\,$. Indeed, the sums in 
\eqref{p} and \eqref{q} are invariant under any permutation of
the indices $1\ts\lc N$ with the corresponding permutation of %the indices
$-\ts1\,,\ldots,-\ts N$ at the same time. 
So $H(u)$ and $C\ts(u\,,\ns v)$
are also invariant under any such permutation. While the vector
\begin{equation}
\label{onedim}
e_{\ts 1}\ot\ldots\ot e_{\ts N}\ot e_{\ts 1}\ot\ldots\ot e_{\ts N}
\end{equation}
changes under the permutation,
its image by the action of $H(u)\ot H(1-N-u)$ does not change
due to Lemma \ref{L2}. Finally, note that 
the product of $C\ts(u\,,\ns v)$ by %the factor 
$H(u)\ot H(\ts1-N-u\ts)\ot1$ on the left is also divisible by
this factor on the right. This remark follows 
by setting $n=N$ in the beginning of the proof of Lemma~\ref{L3}.

For $j=1$ our proposition immediately follows from
Corollary \ref{C5}. For $j=-\ts1$ we will derive
the proposition from Corollary \ref{C6}. We will show that
the summands with the index $k$ in the second and
fifth lines of \eqref{fivel} cancel each other by $\sim$
and so do the summands with the index $k$ in the third and
fourth lines of \eqref{fivel}.

Consider the second and the fifth lines of \eqref{fivel}. 
For $k=1\lc N$ by Lemma~\ref{L2} %we get the relations
\begin{gather*}
-\ts\,e_{\ts1}\ot\ldots\ot e_{\ts k-1}
\ot e_{\ts-1}\ot 
e_{\ts k+1}\ot\ldots\ot e_{\ts N}\ot 
e_{\ts 1}\ot\ldots\ot e_{\ts N}\ot e_{\ts k}\ts\,+
\\[8pt]
e_{\ts 1}\ot\ldots\ot e_{\ts N}\ot
e_{\ts-k}\ot e_{\ts2}\ot\ldots\ot e_{\ts N}\ot e_{\ts k}\,\sim
\\[6pt]
-\ts\,d_{\ts k}\ot e_{\ts-1}\ot e_{\ts 1}\ot 
d_{\ts 1}\ot e_{\ts k}\,(-1)^{\ts N+k}+
d_{\ts k}\ot e_{\ts k}\ot e_{\ts-k}\ot d_{\ts 1}\ot e_{\ts k}\,
(-1)^{\ts N+k}\ts\,\sim
\\[6pt]
d_{\ts k}\ot\ts x\ts\ot d_{\ts 1}\ot e_{\ts k}\,(-1)^{\ts N+k}
\end{gather*}
where $x$ is the vector~\eqref{qimage}.
%and we use the notation~\eqref{dk}. 
So the last displayed line lies in the image of $Q_{\ts N,\ts N+1}\,$.

\newpage%%%%%%%%%%%%%%%%%%%%%%%%%%%%%%%%%%%%%%%%%%%%%%%%%%%%%%%%%%%%%%%%%%%%%%%

Consider the third and the fourth lines of \eqref{fivel}. 
For their summands with the index $k=1$
by Lemma \ref{L2} we get the relations
\begin{gather*}
-\ts\,e_{\ts-1}\ot e_{\ts2}\ot\ldots\ot e_{\ts N}\ot 
e_{\ts 1}\ot\ldots\ot e_{\ts N}\ot e_{\ts 1}\,\ts+
\\[8pt]
e_{\ts 1}\ot\ldots\ot e_{\ts N}\ot
e_{\ts-1}\ot 
e_{\ts 2}\ot\ldots\ot e_{\ts N}\ot e_{\ts 1}\,\sim
\\[6pt]
-\ts\,d_{\ts 1}\ot e_{\ts-1}\ot
e_{\ts 1}\ot d_{\ts 1}\ot e_{\ts 1}\,(-1)^{\ts N+1}+
d_{\ts 1}\ot e_{\ts 1}\ot e_{\ts-1}\ot 
d_{\ts 1}\ot e_{\ts 1}\,(-1)^{\ts N+1}\,\sim
\\[6pt]
d_{\ts 1}\ot x\ot
d_{\ts 1}\ot e_{\ts 1}\,(-1)^{\ts N+1}
\end{gather*}
where $x$ is as above.
Thus the last displayed line lies in the image of $Q_{\ts N,\ts N+1}\,$.

So far in the proof of our proposition 
we used the equality in Lemma \ref{L2} only in the cases 
when both sides of the equality annihilate a given vector. Now 
we will also use Corollary~\ref{C7} that stands outside of these cases. 
Consider the summands in the third and
fourth lines of \eqref{fivel} with any index $k>1\,$.
We get the relations 
\begin{gather*}
-\ts\,e_{\ts-k}\ot e_{\ts2}\ot\ldots\ot e_{\ts N}\ot 
e_{\ts 1}\ot\ldots\ot e_{\ts N}\ot e_{\ts k}\,\ts+
\\[8pt]
e_{\ts 1}\ot\ldots\ot e_{\ts N}\ot
e_{\ts1}\ot\ldots\ot e_{\ts k-1}
\ot e_{\ts-1}\ot 
e_{\ts k+1}\ot\ldots\ot e_{\ts N}\ot e_{\ts k}\,\ts\sim
\\[6pt]
-\ts\,a(u)\,d_{\ts k}\ot e_{\ts-1}\ot 
e_{\ts 1}\ot\ldots\ot e_{\ts N}\ot e_{\ts k}
\,(-1)^{\ts N+k}\,\,+
\\[6pt]
b\ts(u)\, d_{\ts 1}\ot e_{\ts-k}\ot
e_{\ts 1}\ot\ldots\ot e_{\ts N}\ot e_{\ts k}\,
(-1)^{\ts N}\ts\,+
\\[8pt]
a(u)\,e_{\ts 1}\ot\ldots\ot e_{\ts N}\ot
e_{\ts-k}\ot d_{\ts 1}\ot e_{\ts k}
\\[6pt]
-\ts\,b\ts(u)\,
e_{\ts 1}\ot\ldots\ot e_{\ts N}\ot e_{\ts-1}\ot d_{\ts k}\ot e_{\ts k}\,
(-1)^{\ts k}\,\ts\sim
\\[6pt]
-\ts\,a(u)\,d_{\ts k}\ot e_{\ts-1}\ot 
e_{\ts 1}\ot d_{\ts 1}\ot e_{\ts k}
\,(-1)^{\ts N+k}
\\[6pt]
-\ts\,b\ts(u)\, d_{\ts 1}\ot e_{\ts-k}\ot
e_{\ts k}\ot d_{\ts k}\ot e_{\ts k}\,
(-1)^{\ts N+k}+
\\[6pt]
a(u)\,
d_{\ts k}\ot e_{\ts k}\ot e_{\ts-k}\ot d_{\ts 1}\ot e_{\ts k}\,
(-1)^{\ts N+k}\,\,+
\\[6pt]
b\ts(u)\,d_{\ts1}\ot e_{\ts 1}\ot e_{\ts-1}\ot d_{\ts k}\ot e_{\ts k}\,
(-1)^{\ts N+k}\,\ts\sim
\\[6pt]
a(u)\,
d_{\ts k}\ot x\ot d_{\ts1}\ot e_{\ts k}\,
(-1)^{\ts N+k}\,+
b\ts(u)\,
d_{\ts 1}\ot x\ot d_{\ts k}\ot e_{\ts k}\,
(-1)^{\ts N+k}
\end{gather*}
where $x$ is still as above.
So the last displayed line lies in the image of $Q_{\ts N,\ts N+1}\,$.

Note that Corollary \ref{C7}
remains valid for $k=1\,$. In this particular case
the two linear combinations in the corollary are equal to
$(-1)^{\ts N+1}\,d_{\ts1}\ot e_{\ts-1}$ and 
$e_{\ts-1}\ot d_{\ts1}$ respectively, 
because $a(u)+b\ts(u)=1\,$.
The arguments used the last paragraph here also remain valid for $k=1\,$,
but simplify as in the paragraph before.
\qed
\end{proof}

%------------------------------------------------------------------------------

\vspace{-\baselineskip}
\refstepcounter{section}
\subsection*{\bf\thesection}
\label{22}

Let us employ Proposition \ref{P8} in the proof of Theorem \ref{T3}. 
Consider the algebra
\begin{equation}
\label{algebra}
(\End\CN)^{\ts\ot\ts(2\ts N+1)}\ot\XpN\,.
\end{equation}
Elements of this algebra can be applied to
vectors of $(\CN)^{\ts\ot\ts(2\ts N+1)}$ via
$2\ts N+1$ tensor factors $\End(\CN)\,$.
Equip the vector space $(\CN)^{\ts\ot\ts(2\ts N+1)}$ with the %standard 
basis of tensor products of the 
vectors $e_{\ts i}\in\CN\,$. This basis includes the vectors \eqref{fj}.

We have an equality of formal power series in $u^{-1},v^{-1}$
with coefficients in~\eqref{algebra}
\begin{gather*}
(\ts H(u)\ot H(\ts1-N-u\ts)\ot1\ot1\ts)\,
(\ts C\ts(u\,,\ns v)\ot1\ts)\ts\,\times
\\[4pt]
T_{\ts1}(u+N-1)\,\ldots\,T_{\ts N}(u)\,
T_{\ts N+1}(-\ts u)\,\ldots\,T_{\,2\ts N}(\ts1-N-u\ts)\,
T_{\,2\ts N+1}(v)=
\\[4pt]
(\ts H(u)\ot H(\ts1-N-u\ts)\ot1\ot1\ts)\,
T_{\,2\ts N+1}(v)\ts\,\times
\\[4pt]
T_{\ts1}(u+N-1)\,\ldots\,T_{\ts N}(u)\,
T_{\ts N+1}(-\ts u)\,\ldots\,T_{\,2\ts N}(\ts1-N-u\ts)\,
(\ts C\ts(u\,,\ns v)\ot1\ts)\,,
\end{gather*}
see the beginning of Section \ref{19}.
For any indices $i\,,j$ apply the right hand side of this equality
to the basis vector $f_{\ts j}$ and take 
the coefficient of the result at the basis vector $f_{\ts i}\,$.
By Proposition \ref{P8} and by
the second formula for $A\ts(u)$ as given in Section~\ref{15}
this coefficient is equal to
\begin{equation}
\label{TB}
\frac{(z-1)\,(w+1)}{z\,w}\,\,
T_{\ts ij}(v)\,A\ts(u)\,A\ts(\ts1-N-u\ts)\,.
\end{equation}
Indeed, due to \eqref{3.3} the product 
$$
T_{\ts1}(u+N-1)\,\ldots\,T_{\ts N}(u)\,
T_{\ts N+1}(-\ts u)\,\ldots\,T_{\,2\ts N}(\ts1-N-u\ts)
$$
by $H(u)\ot H(\ts1-N-u\ts)\ot1\ot1$ on the left 
is divisible by the same factor on~the right,
see the beginning of Section \ref{16}. 
Due to \eqref{QQ}
the product of 
$$
T_{\ts N}(u)\,T_{\ts N+1}(-\ts u)\,(\ts Q_{\ts N,\ts N+1}\ot1\ts)
$$ 
is divisible by $Q_{\ts N,\ts N+1}\ot1$ on the left. 
We use the relation
$Q_{\ts N+1,\ts N}=-\ts\,Q_{\ts N,\ts N+1}$ which follows from 
%the first two equalities in 
\eqref{pq}.
But the coefficient at $f_{\ts i}$ of any vector from the image of 
the action of
$(\ts H(u)\ot H(\ts1-N-u\ts)\ot 1\ts)\,Q_{\ts N,\ts N+1}$
on $(\CN)^{\ts\ot\ts(2\ts N+1)}$ is zero.
 
Next we shall prove that by 
applying the left hand side of the above displayed equality
to the vector $f_{\ts j}$ and then taking
the coefficient of the result at $f_{\ts i}$ we~get
\begin{equation}
\label{BT}
\frac{(z-1)\,(w+1)}{z\,w}\,\,
A\ts(u)\,A\ts(\ts1-N-u\ts)\,T_{\ts ij}(v)\,.
\end{equation}
Thus the product \eqref{TB} equals \eqref{BT}.
These equalities for all $i\,,j$ imply Theorem~\ref{T3}. 
Considering the
left hand side of the above displayed equality will be easier than the
right hand side,
thanks to the obvious advantages of Lemma~\ref{L1}~over Lemma~\ref{L2}. 
%In particular, we we will use our Lemma \ref{L3}.

%------------------------------------------------------------------------------

\vspace{-\baselineskip}
\refstepcounter{section}
\subsection*{\bf\thesection}
\label{23}

Let us now use the right action of the $\ZZ_{\ts2}$-graded
algebra $\End(\CN)$ on~the $\ZZ_{\ts2}$-graded vector space
$\CN$ instead of the left action which we used before. 
For the %standard 
matrix units $E_{\ts ij}\in\End\CN$ we have
$e_{\ts k}\,E_{\ts ij}=\de_{\ts ik}\,e_{\ts j}\,$. For
for any positive integer 
$n$ we can then define the right action
of the algebra $(\End\CN)^{\ot\ts n}$ 
on the vector space $(\CN)^{\ot\ts n}$ by the conventions similar to
\eqref{XY} and \eqref{xy}. %We omit these details.

Notice that the right action of the element $P\in(\End\CN)^{\ot\ts 2}$
defined by~\eqref{p} still maps 
$
e_i\ot e_{j}\mapsto e_j\ot e_i\,{(-1)}^{\,\bi\ts\bj}\ts.
$
The image of the right action on $(\CN)^{\ts\ot\ts2}$
of the element $Q$ in \eqref{q} is still one-dimensional, but is spanned by 
the vector
\begin{equation}
\label{qqimage}
\sum_j\,e_{\ts j}\ot e_{\,-\ts j}\,.
%\vspace{-2pt}
\end{equation}
By using Lemmas \ref{L1} and \ref{L3} with $n=N$ 
we immediately get the next two lemmas.

\begin{lemma}
\label{L6}
Applying $(\ts H(u)\ot1\ts)\,F\ts(u\,,\ns v)$ to %the vector
$e_{\ts 1}\ot\ldots\ot e_{\ts N}\ot e_{\ts1}$ from the right
has the same effect as
applying $H(u)\ot1$ from the right to the vector~\eqref{v4}.
\end{lemma}

\begin{lemma}
\label{L7}
Applying $(\ts H(u)\ot1\ts)\,F\ts(u\,,\ns v)$ to %the vector
$e_{\ts 1}\ot\ldots\ot e_{\ts N}\ot e_{\ts-1}$ from the right
has the same effect as
applying $H(u)\ot1$ from the right to the linear combination
\begin{gather*}
\frac{w+1}w\,\,e_{\ts 1}\ot\ldots\ot e_{\ts N}\ot e_{\ts-1}
\\[0pt]
-\,\frac1z\,\,\sum_{k=1}^N\,\ts
e_{\ts1}\ot\ldots\ot e_{\ts k-1}
\ot e_{\ts-1}\ot 
e_{\ts k+1}\ot\ldots\ot e_{\ts N}\ot e_{\ts k}\ts\,+
\\[-5pt]
\frac1w\,\,\sum_{k=1}^N\,\ts
e_{\ts-k}\ot e_{\ts2}\ot\ldots\ot e_{\ts N}\ot e_{\ts k}\,.
\end{gather*}
\end{lemma}

By using Lemma \ref{L6} first in its original formulation
and then with the variable $u$ replaced by $1-N-u$
we immeditaly get the following corollary to this lemma.

\begin{corollary}
\label{C8}
Applying $(\ts H(u)\ot H(\ts1-N-u\ts)\ot1\ts)\,C\ts(u\,,\ns v)$ to %the vector
$f_{\ts 1}$ from the right has the same effect as applying 
$H(u)\ot H(\ts1-N-u\ts)\ot1$ from the right to %the vector 
\eqref{v5}.
\end{corollary}

Below is the counterpart of Corollary \ref{C8}
for our Lemma \ref{L7} instead of Lemma~\ref{L6}.

\begin{corollary}
\label{C9}
Applying $(\ts H(u)\ot H(\ts1-N-u\ts)\ot1\ts)\,C\ts(u\,,\ns v)$ to %the vector
$f_{\ts-1}$ from the right has the same effect 
as applying $H(u)\ot H(\ts1-N-u\ts)\ot1$ from the right to
%the linear combination
\begin{gather}
\notag
\hspace{-10pt}
\frac{(z-1)\,(w+1)}{z\,w}\,f_{\ts-1}\,\ts+
\\[0pt]
\notag
\hspace{10pt}
\frac{w+1}{w^{\ts2}}\,\,\sum_{k=1}^N\,\ts
e_{\ts 1}\ot\ldots\ot e_{\ts N}\ot
e_{\ts1}\ot\ldots\ot e_{\ts k-1}
\ot e_{\ts-1}\ot 
e_{\ts k+1}\ot\ldots\ot e_{\ts N}\ot e_{\ts k}
\\[-5pt]
\notag
\hspace{36pt}
-\,\ts\frac{w+1}{z\,w}\,\,\sum_{k=1}^N\,\ts
e_{\ts 1}\ot\ldots\ot e_{\ts N}\ot
e_{\ts-k}\ot e_{\ts2}\ot\ldots\ot e_{\ts N}\ot e_{\ts k}
\\[-5pt]
\notag
\hspace{-10pt}
-\,\ts\frac{w+1}{z\,w}\,\,\sum_{k=1}^N\,\ts
e_{\ts1}\ot\ldots\ot e_{\ts k-1}
\ot e_{\ts-1}\ot 
e_{\ts k+1}\ot\ldots\ot e_{\ts N}\ot 
e_{\ts 1}\ot\ldots\ot e_{\ts N}\ot e_{\ts k}\,\ts+
\\[-5pt]
\label{fiver}
\hspace{-10pt}
\qquad\qquad
\frac{w+1}{w^{\ts2}}\,\,\sum_{k=1}^N\,\ts
e_{\ts-k}\ot e_{\ts2}\ot\ldots\ot e_{\ts N}\ot 
e_{\ts 1}\ot\ldots\ot e_{\ts N}\ot e_{\ts k}\,.
\end{gather}
\end{corollary}

\begin{proof}
Using Lemma \ref{L7} we obtain that applying
$(\ts H(u)\ot1\ot1\ts)\,F^{\,\prime}\ts(\ts u\,,\ns v)$ to %the vector
$f_{\ts-1}$ from the right has the same effect as applying $H(u)\ot1\ot1$ 
from the right~to %the linear combination 
\begin{gather}
\notag
\frac{w+1}w\,f_{\ts-1}
\\[0pt]
\notag
-\,\frac1z\,\,\sum_{k=1}^N\,\ts
e_{\ts1}\ot\ldots\ot e_{\ts k-1}
\ot e_{\ts-1}\ot 
e_{\ts k+1}\ot\ldots\ot e_{\ts N}\ot 
e_{\ts 1}\ot\ldots\ot e_{\ts N}\ot e_{\ts k}\ts\,+
\\[-5pt]
\label{threerines}
\frac1w\,\,\sum_{k=1}^N\,\ts
e_{\ts-k}\ot e_{\ts2}\ot\ldots\ot e_{\ts N}\ot 
e_{\ts 1}\ot\ldots\ot e_{\ts N}\ot e_{\ts k}\,.
\end{gather}

By Lemma \ref{L7} with the variable $u$ replaced by $1-N-u\,$,
applying the product
$(\ts1\ot H(\ts1-N-u\ts)\ot1\ts)\,F^{\,\prime\prime}\ts(\ts1-N-u\,,\ns v)$
to the first of three lines \eqref{threerines} from the right
has the same effect as applying $1\ot H(\ts1-N-u\ts)\ot1$ from the right~to 
the sum of first three lines in~\eqref{fiver}.

Applying 
$(\ts1\ot H(\ts1-N-u\ts)\ot1\ts)\,F^{\,\prime\prime}\ts(\ts1-N-u\,,\ns v)$
to the second and to the third lines of \eqref{threerines} from the right
has the same effect as applying from the~right %the product
\begin{equation*}
\biggl(\,1+\frac{P_{\,2\ts N,2\ts N+1}}{u+v+N-1}\biggr)
\ts\ldots\ts
\biggl(\,1+\frac{P_{\,N+1,2\ts N+1}}{u+v}\,\biggr)\,
(\ts1\ot H(\ts1-N-u\ts)\ot1\ts)\,,
\end{equation*}
see the beginning of the proof of Lemma \ref{L3} for $n=N\,$.
By Lemma \ref{L1} applying the last displayed product
to the second and the third lines of \eqref{threerines} from the right
gives the same result as applying to them from the right the product
\begin{equation}
\label{numerar}
\biggl(
\,1+\frac{P_{\,N+1,2\ts N+1}+\ldots+P_{\,2\ts N,2\ts N+1}}{u+v+N-1}
\biggr)\,
(\ts1\ot H(\ts1-N-u\ts)\ot1\ts)
\,.
\end{equation}
When applying \eqref{numerar}
to the summands with the index $k$ in the second and third lines of
\eqref{threerines}, 
all summands of the numerator in \eqref{numerar}
vanish except $P_{\,N+k\ts,\ts2\ts N+1}\,$.
The latter does not change the summands with the index $k$
in the second and third lines of \eqref{threerines}. 
Hence applying 
$(\ts1\ot H(\ts1-N-u\ts)\ot1\ts)\,F^{\,\prime\prime}\ts(\ts1-N-u\,,\ns v)$
to the second and the third lines of \eqref{threerines}
from the right gives the same as applying 
$1\ot H(\ts1-N-u\ts)\ot1$ %from the right
respectively to the fourth and the fifth lines of \eqref{fiver}.
\qed
\end{proof}

%------------------------------------------------------------------------------

\vspace{-\baselineskip}
\refstepcounter{section}
\subsection*{\bf\thesection}
\label{24}

Changing the notation of Section \ref{21}
let us now relate by the symbol $\sim$
two vectors in $(\CN)^{\ts\ot\ts(2\ts N+1)}(u\,,v)$ to each other
if either their difference vanishes under the right action of
$H(u)\ot H(1-N-u)\ot1\,$, 
or it lies in the image of the right action of $Q_{\,1,\ts2\ts N}\,$.
Extend the new relation $\sim$ transitively. Denote
\begin{gather*}
D\ts(u\,,\ns v)=
R_{\,N,2\ts N+1\ts}(\ts u\,,\ns v\ts)\,\ldots\,
R_{\,1,2\ts N+1\ts}(\ts u+N-1\,,\ns v\ts)\,\times
\\[4pt]
R_{\,2\ts N,2\ts N+1\ts}(\ts 1-u-N\,,\ns v\ts)\,\ldots\,
R_{\,N+1,2\ts N+1\ts}(-\,u\,,\ns v\ts)\,\,.
\end{gather*}
By choosing $n=N$ in the beginning of the proof Lemma \ref{L3}
we obtain the equality
\begin{gather}
\notag
(\ts H(u)\ot H(\ts1-N-u\ts)\ot1\ts)\ts\,C\ts(u\,,\ns v)=
\\[4pt]
\label{cuv}
D\ts(u\,,\ns v)\,(\ts H(u)\ot H(\ts1-N-u\ts)\ot1\ts)\,.
\end{gather}

\begin{proposition}
\label{P9}
For\/ $i=\pm\,1\lc\!\pm\ns N$ and for
the right action on the vector $f_{\ts i}$
$$
f_{\ts i}\,D\ts(u\,,\ns v)\,\sim\,\frac{(z-1)\,(w+1)}{z\,w}\,f_{\ts i}\,.
$$
\end{proposition}

\begin{proof}
Proposition \ref{P9} follows from its particular cases of $i=1\,,-\ts1$ 
by the same remark as already used in the beginning of the proof of 
Proposition \ref{P8}.
For $i=1$ our proposition follows from
Corollary \ref{C8}, see the equality \eqref{cuv} above. 
For $i=-\ts1$ we will derive Proposition \ref{P9}
from Corollary \ref{C9}. We will prove that
the summands with the index $k$ in the second and
fifth lines of \eqref{fiver} cancel each other by $\sim$
and that so do the summands with the index $k$ in the third and
fourth lines of \eqref{fiver}. 

Consider the second and the fifth lines of \eqref{fiver}. 
For $k=1\lc N$ by Lemma~\ref{L1} %we get the relations
\begin{gather*}
e_{\ts 1}\ot\ldots\ot e_{\ts N}\ot
e_{\ts1}\ot\ldots\ot e_{\ts k-1}
\ot e_{\ts-1}\ot 
e_{\ts k+1}\ot\ldots\ot e_{\ts N}\ot e_{\ts k}\ts\,+
\\[8pt]
e_{\ts-k}\ot e_{\ts2}\ot\ldots\ot e_{\ts N}\ot 
e_{\ts 1}\ot\ldots\ot e_{\ts N}\ot e_{\ts k}\,\sim
\\[6pt]
e_{\ts 1}\ot d_{\ts1}\ot d_{\ts k}\ot e_{\ts-k}\ot e_{\ts k}
\,(-1)^{\ts N+k}+
e_{\ts-k}\ot d_{\ts1}\ot d_{\ts k}\ot e_{\ts k}\ot e_{\ts k}\,
(-1)^{\ts N+k}\ts\,\sim
\\[4pt]
\sum_j\,
e_{\ts j}\ot d_{\ts1}\ot d_{\ts k}\ot e_{\ts-j}\ot e_{\ts k}
\,(-1)^{\ts N+k}\,.
\end{gather*}
%and we use the notation~\eqref{dk}. 
So the last displayed line lies in the image of 
the right action of $Q_{\,1,\ts2\ts N}\,$.
Here we use the notation introduced just before stating
Corollary \ref{C7}.

Now consider the third and fourth lines of \eqref{fivel}. 
For $k=1\lc N$ by Lemma~\ref{L1} %we get the relations
\begin{gather*}
e_{\ts 1}\ot\ldots\ot e_{\ts N}\ot
e_{\ts-k}\ot e_{\ts2}\ot\ldots\ot e_{\ts N}\ot e_{\ts k}\ts\,+
\\[8pt]
e_{\ts1}\ot\ldots\ot e_{\ts k-1}
\ot e_{\ts-1}\ot 
e_{\ts k+1}\ot\ldots\ot e_{\ts N}\ot 
e_{\ts 1}\ot\ldots\ot e_{\ts N}\ot e_{\ts k}\ts\,\sim
\\[6pt]
e_{\ts k}\ot d_{\ts k}\ot d_{\ts 1}\ot e_{\ts-k}\ot e_{\ts k}
\,(-1)^{\ts N+k}+
e_{\ts-1}\ot d_{\ts k}\ot d_{\ts 1}\ot e_{\ts 1}\ot e_{\ts k}\,
(-1)^{\ts N+k}\ts\,\sim
\\[4pt]
\sum_j\,
e_{\ts j}\ot d_{\ts k}\ot d_{\ts 1}\ot e_{\ts-j}\ot e_{\ts k}
\,(-1)^{\ts N+k}\,.
\end{gather*}
So the last displayed line also lies in the image of 
the right action of $Q_{\,1,\ts2\ts N}\,$.
\qed
\end{proof}

The vector space $(\CN)^{\ts\ot\ts2}$ has a
basis of the vectors $e_{\ts i}\ot e_{\ts j}\,$.
Applying~any $X\in(\End\CN)^{\ts\ot\ts2}$
to a basis vector from the left and taking the coefficient 
of the result at another 
basis vector is not always the same as applying $X$ to the latter vector from
the right and taking the coefficient of the result at 
the former vector. But these two coefficients will be
the same for $X=P$ and for $X=Q\,$. 

Indeed, both actions of $P$
map $e_i\ot e_{j}\mapsto e_j\ot e_i\,{(-1)}^{\,\bi\ts\bj}\ts$.
Further, applying $Q$ to the vector 
$e_{\ts j}\ot e_{\,-\ts j}$ from the left
yields the sum \eqref{qimage}. Applying $Q$ to the vector
$e_{\ts i}\ot e_{\,-\ts i}$ in \eqref{qimage}
from the right yields the sum \eqref{qqimage}
multiplied by~${(-1)}^{\,\bi}$.

It follows that applying the product
$(\ts H(u)\ot H(\ts1-N-u\ts)\ot1\ts)\,C\ts(u\,,\ns v)$
to any basis vector of $(\CN)^{\ts\ot\ts(2\ts N+1)}$
from the left and taking the coefficient 
of the result at another 
basis vector is the same as applying this product
to the latter vector from
the right and taking the coefficient of the result at 
the former~vector.

Let us now consider the product 
at the left hand side of the displayed equality 
in Section \ref{22}. To complete our proof
Theorem \ref{T3} we have to show that 
for any indices $i\,,j$ applying this product  
to the basis vector $f_{\ts j}$ and taking 
the coefficient of the result at $f_{\ts i}$
gives \eqref{BT}. But this follows from Proposition \ref{P9}
due to the above remark about the left and right actions on the basis
vectors of $(\CN)^{\ts\ot\ts(2\ts N+1)}\,$.

Here we use second formula for $A(u)$ in Section \ref{15}
and the equality \eqref{cuv}.
We also use the equality displayed just before stating Lemma \ref{L1}
with $n=N\ts$. And we also use the following observation. 
Due to \eqref{QQ} the product 
$$
(\ts Q_{\ts 1,\ts2\ts N}\ot1\ts)\,
T_{\ts 1}(u+N-1)\,T_{\ts2\ts N}(\ts1-N-u\ts)
$$ by
is divisible by $Q_{\ts 1,\ts2\ts N}\ot1$ on the right.
But the coefficient at $f_{\ts j}$ of any vector from the image of 
the right action of
$
Q_{\ts 1,\ts2\ts N}\,(\ts H(u)\ot H(\ts1-N-u\ts)\ot 1\ts)
$
on $(\CN)^{\ts\ot\ts(2\ts N+1)}$ is zero.
Our proof of Theorem \ref{T3} is now complete. 

%------------------------------------------------------------------------------

\vspace{-\baselineskip}
\refstepcounter{section}
\subsection*{\bf\thesection}
\label{25}

In this section we will provide an interpretation of 
Proposition~\ref{P8}. %Accrodinly to that proposition
We will employ the usual left action of the algebra $\End\CN$ on
the vector space $\CN\,$. 
Proposition \ref{P9} admits a similar interpretation.
But then one has to use the right action of the algebra $\End\CN$ on
the vector space $\CN$ instead of left action. 

For any $t\in\CC$ consider 
the representation $\rho_{\,t}\ot\rho_{\,-\ts t}$ of the algebra $\XpN\,$.
Using the comultiplication \eqref{3.7} %on $\XpN$
this representatioin is determined by mapping
\begin{gather*}
(\End\CN)\ot\XpN\ts[[u^{-1}]]\to(\End\CN)^{\ot\ts3}\ts[[u^{-1}]]:
\\[4pt]
T(u)\mapsto R_{\ts12}(u\,,t\ts)\,R_{\ts13}(u\,,-\,t\ts)\,.
\end{gather*}
See the end of our Section \ref{6} for the definition of the representation
$\rho_{\,t}$ of $\XpN\,$.

Multiplying the relation \eqref{rrr} by $v+w$ and then setting
$v=-\,w=t$ yields
$$
R_{\ts12}(u\,,t\ts)\,R_{\ts13}(u\,,-\,t\ts)\,Q_{\ts23}=
Q_{\ts23}\,R_{\ts13}(u\,,-\,t\ts)\,R_{\ts12}(u\,,t\ts)\,.
$$
The latter relation shows that that under the representation
$\rho_{\,t}\ot\rho_{\,-\ts t}$ the action of 
the algebra $\XpN$ preserves the %one-dimensional 
subspace $\im\,Q\subset(\CN)^{\ts\ot\ts2}\,$.

Suppose that $-\,2\,t\neq 1,2\lc2\ts N-3$ if $N>1\,$.
Then  the rational functions $H(u)$ and $H(\ts1-N-u\ts)$
have no poles at $u=t\,$. For $N=1$ we get $H(u)=1$ and 
the complex number $t$ is still arbitrary. For any $N\ge1$ let
$W\in(\ts\End\CN\ts)^{\ot\ts2\ts N}$ be the value of the function
$H(u)\ot H(\ts1-N-u\ts)$ at~$u=t\,$.

Now consider the representation of $\XpN$ on the vector space 
$(\CN\ts)^{\ot\ts2\ts N}$
$$
\rho=\rho_{\,1-N-\ts t}\ot\ldots\ot\rho_{\,-\ts t}\ot
\rho_{\,t}\ot\ldots\ot \rho_{\,t\ts+N-1}\,.
$$
For this particular choice of the representation $\rho$ we have
\begin{gather*}
\id\ot\rho:
(\End\CN)\ot\XpN\ts[[v^{-1}]]\to(\End\CN)^{\ot\ts2\ts N+1}\ts[[v^{-1}]]:
\\[4pt]
T(-\,v)\mapsto
R_{\,12}(-\,v\,,1-N-\ts t\ts)\,\ldots\,R_{\,1,N+1}(-\,v\,,-\,t\ts)
\,\times
\\[4pt]
R_{\,1,N+2}(-\,v\,,t\ts)\,\ldots\,R_{\,1,2\ts N+1}(-\,v\,,t\ts+N-1\ts)
\,=
\\[4pt]
P_{\,12}\,\ldots\,P_{\,2\ts N,2\ts N+1}\,
R_{\,2\ts N+1,1}(-\,v\,,1-N-\ts t\ts)
\,\ldots\,
R_{\,2\ts N+1,N}(-\,v\,,-\,t\ts)
\,\times
\\[4pt]
R_{\,2\ts N+1,N+1}(-\,v\,,t\ts)
\,\ldots\,
R_{\,2\ts N+1,2\ts N}(-\,v\,,t\ts+N-1\ts)\,
P_{\,2\ts N,2\ts N+1}\,\ldots\,P_{\,12}\,=
\\[4pt]
P_{\,12}\,\ldots\,P_{\,2\ts N,2\ts N+1}\,
C\ts(\,t\,,\ns v\ts)\,
P_{\,2\ts N,2\ts N+1}\,\ldots\,P_{\,12}\,.
\end{gather*}
The last equality is obtained by using %the relation 
\eqref{prp}.
The equality \eqref{cuv}
now shows that the action of
the algebra $\XpN$ by $\rho$ preserves the subspace
$\ker\ts W\subset(\CN)^{\ts\ot\ts2\ts N}\,$.
This action also preserves the subspace 
$\im\ts\,Q_{\ts N,\ts N+1}\subset(\CN)^{\ts\ot\ts2\ts N}$
by the above remark. Proposition \ref{P8}
implies that this action preserves the 
subspace spanned by $\ker\ts W$, $\im\ts\,Q_{\ts N,\ts N+1}$
and the vector \eqref{onedim}.
Moreover under the action of the algebra 
$\XpN$ on the one-dimensional
subquotient of $(\CN)^{\ts\ot\ts2\ts N}$
spanned by~\eqref{onedim} %in such a way so that
\begin{equation*}
%\label{reciproc}
T_{\ts ij}(-\,v)\ts\mapsto\ts\de_{\ts ij}\,\ts
\frac{(\ts t-v-1\ts)\,(\ts t+v+N\ts)}{(\ts t-v\ts)\,(\ts t+v+N-1\ts)}\,\,.
\end{equation*}

Let us choose $c\ts(u)$ in \eqref{muf} so that $c\ts(-\,v)$
is the reciprocal of the last displayed fraction. 
Then $c\ts(u)\,T_{\ts ij}(u)$ acts on our one-dimensional
subquotient just as $\de_{\ts ij}\,$. 
%Direct calculation shows that 
For this $c\ts(u)$ the composition of 
the representation $\rho$ with the automorphism \eqref{muf} of $\XpN$ 
maps $Z(u)\mapsto1\,$, see Section~\ref{11}.
Hence the composite representation of $\XpN$ on $(\CN)^{\ts\ot\ts2\ts N}$
factors through $\YpN\,$.

%------------------------------------------------------------------------------

\refstepcounter{section}
\subsection*{\bf\thesection}
\label{26}

There is a natural 
$\ZZ\ts$-grading on the associative algebra $\YpN\,$.
It is defined by setting the degree of the generator
$T^{\ts(r)}_{ij}$ to $0$ if the indices $i\,,j$ are of the same sign,
to $1$ if $i>0>j$ and to $-\ts1$ if $i<0<j\,$.
Therefore twice the $\ZZ\ts$-degree of $T^{\ts(r)}_{ij}$
is %equal to 
the eigenvalue of the %corresponding
element $E_{\ts ij}\in\End\CN$ under the adjoint action~of %the element
$$
E=E_{\ts 11}-E_{\ts-1\ts,\ts-1}+\ldots+E_{\ts NN}-E_{\ts-N\ts,\ts-N}\,.
$$
The elements \eqref{p} and \eqref{q} of the algebra 
$(\End\CN)^{\ot\ts2}$ both
commute with the element $E\ot1+1\ot E\,$.
Hence the defining relations of the algebra 
$\YpN$ are homogeneous, see \eqref{rttttr} and \eqref{CC1} with $Z(u)=1\ts$.
So our $\ZZ\ts$-grading is well defined.
We will use this grading to prove that under the antipodal map
\eqref{S} on $\YpN$
\begin{equation}
\label{bin}
B(u)\mapsto B(u)^{-1}\,.
\end{equation}

Let $\mathrm{I}$ be the subspace of $\YpN$
consisting of all elements of $\ZZ\ts$-degree $0\,$. 
This subspace is a subalgebra. Using the relations \eqref{rttttr}
each element of $\YpN$ can be written as a linear 
combination of monomials in the generators $T^{\ts(r)}_{ij}$ 
where all factors of degree $1$ are on the left and
all factors of degree $-\ts1$ are on the right.
If such a monomial belongs to $\mathrm{I}\ts\setminus\ns\{0\}$
and has a factor of degree $1\ts$,
it also has a factor of degree $-\ts1\ts$.
So the monomials from $\mathrm{I}\ts\setminus\ns\{0\}$
span a two-sided ideal of~$\mathrm{I}\,$.%the algebra

Denote this ideal by $\mathrm{J}\,$. 
Take any monomial in the generators $T^{\ts(r)}_{ij}$ of degree~$0\ts$.
By using the relations \eqref{rttttr} and \eqref{CC1} with $Z(u)=1\ts$,
this monomial can be written as a linear 
combination of monomials in the generators $T^{\ts(r)}_{ij}$ 
with $i\,,j>0$ only and of some monomials from $\mathrm{J}\,$. 
So the quotient algebra $\mathrm{I}\,/\,\mathrm{J}$
can be identified with the subalgebra of $\YpN$
in Corollary \ref{C3}\ts. 
Let $\mathrm{A}$ be the latter subalgebra 
and $\al:\mathrm{I}\to\mathrm{A}$ be the canonical homomorphism
coming from this identification. 

The homomorphism $\al$ maps the series $B(u)$ to itself, 
because all coefficients of this series belong to $\mathrm{A}\,$.
Thus $\al$ maps the series $B(u)^{-1}$ to itself too.
Consider the image of the series $B(u)$ by the 
antipodal map \eqref{S} on $\YpN\,$. This image is a formal power 
series in $u^{-1}$ with coefficients in the centre of $\YpN\,$.
We will show that by applying $\al$ to this image we
get the series $B(u)^{-1}$ as well.
Due to Corollary \ref{C4} this fact will imply the correspondence
\eqref{bin} under the map \eqref{S}.

Let $G\ts(u)$ the $2\ts N\times 2\ts N$ matrix
whose $ij$ entry is the series 
$$
G_{\ts ij}(u)=
T_{ij}(u)\,{(-1)}^{\,\bi\ts\bj\ts+\,\bj}\,.
$$
We index the rows and columns of $G\ts(u)$ by
the numbers $1\lc N,-\ts1\ts,\ldots,-\ts N\ts$.
The matrix $G\ts(u)$ is invertible.
Let $\GP_{\ts ij}(u)$ be the $ij$ entry of the inverse matrix.
Then by the definition \eqref{Tui} for any indices $i$ and $j$
we have the equality
$$
\GP_{\ts ij}(u)=\TP_{ij}(u)\,{(-1)}^{\,\bi\ts\bj\ts+\,\bj}\,.
$$

The matrix $G\ts(u)$ splits to four blocks of size $N\times N$
corresponding to different signs of the indices $i$ and $j\,$. By using for instance \cite[Lemma 3.2]{BF}
we obtain that for $i\,,j>0$ the homorphism $\al$ maps
$\GP_{\ts ij}(u)$ to the $ij$ entry of the %$N\times N$ matrix
inverse of the block of $G\ts(u)$ corresponding to the positive
matrix indices. But for $i\,,j>0$ we have 
$G_{\ts ij}(u)=T_{ij}(u)$ and $\GP_{\ts ij}(u)=\TP_{ij}(u)\,$.
Therefore by the definition \eqref{Au} the homomorphism $\al$ maps
the image of $B(u)$ by \eqref{S} to $B(u)^{-1}$ as needed.
Here we use \cite[Proposition 2.19]{MNO} and our
Corollary \ref{C3}, also see the end of our Section~\ref{9}.

%==============================================================================

\newpage%%%%%%%%%%%%%%%%%%%%%%%%%%%%%%%%%%%%%%%%%%%%%%%%%%%%%%%%%%%%%%%%%%%%%%%

\section*{Acknowledgements}

I am grateful to Jonathan Brundan for his suggestion to employ
\cite{B} in the proof of Theorem \ref{T2}. 
I am also grateful to Inna Entova\ts-Aizenbud and Vera Serganova for 
helpful advice on the representation theory of 
periplectic Lie superalgebras.

%\section*{Declarations}

%The author did not receive support from any organization for this work,
%and has no competing interests to declare that are relevant to the content of %this article. 
%All research data supporting this work are available within this article.

%==============================================================================

%==============================================================================

\end{document}